\documentclass[12pt,a4paper,leqno]{amsart}
\usepackage{amssymb,xspace}
\usepackage{amsthm}
\usepackage{amstext}
\usepackage{pdfsync}
\usepackage{hyperref}
\usepackage[mathscr]{eucal}
\theoremstyle{plain}
\usepackage{amsbsy,amssymb,amsfonts,latexsym,eucal,amscd}
\usepackage[dvips]{graphicx}
\usepackage{epsfig}
\usepackage[all]{xy}
\usepackage{pstricks}
\usepackage{mathrsfs}
\usepackage{tikz}
\usetikzlibrary{shapes}
\usetikzlibrary{positioning}
\usetikzlibrary{arrows}
\usetikzlibrary{calc}
\usepackage{txfonts}
\usepackage{hyperref}
\makeatletter
\def\l@subsection{\@tocline{2}{0pt}{1pc}{5pc}{\hspace{2em}}}
\makeatother

\newcommand{\tens}[1][]{\mathbin{\otimes_{\raise1.5ex\hbox to-.1em{}{#1}}}}

\newcommand{\R }{\ensuremath{\mathbb R}}
\newcommand{\C }{\ensuremath{\mathbb C}}
\newcommand{\Q }{\ensuremath{\mathbb Q}}
\newcommand{\Z }{\ensuremath{\mathbb Z}}

\renewcommand{\P }{\ensuremath{\mathbb P}}

\newcommand{\oo }{\ensuremath{\mathcal{O}}}

\newtheorem{theorem}{Theorem}[section]

\newtheorem{lemma}[theorem]{Lemma}

\newtheorem{proposition}[theorem]{Proposition}

\newtheorem{corollary}[theorem]{Corollary}

{\theoremstyle{definition}}

{\theoremstyle{definition}}

{\theoremstyle{definition}}

{\theoremstyle{definition}}

{\theoremstyle{definition}\newtheorem{definition}[theorem]{Definition}}

{\theoremstyle{definition}}

{\theoremstyle{definition}\newtheorem{remark}[theorem]{Remark}}

\entrymodifiers={+!!<0pt,\fontdimen22\textfont2>}

\def\apl#1#2#3{#1\mkern -1 mu:\mkern - 6 mu
\xymatrix@C=17pt{#2\!\ar[r]&\!#3}
}

\def\aplexp#1#2#3#4{#1\mkern -1 mu:\mkern - 6 mu
\xymatrix@C=17pt{#2\!\ar[r]^-{#4}&\!#3}
}

\def\aplcourte#1#2#3{#1\mkern -4 mu:\mkern - 8 mu
\xymatrix@C=12pt{#2\!\ar[r]&\!#3}
}

\def\aplpt#1#2#3#4{#1\mkern -4 mu:\mkern - 8 mu
\xymatrix@C=17pt{#2\!\ar[r]&\!#3#4}
}

\author{Julien Grivaux}

\address{CNRS  \& Institut de Math\'ematiques de Marseille, Universit\'e d'Aix-Marseille, $39$ rue 
Fr\'ed\'eric Joliot-Curie, $13453$ Marseille Cedex $13$, France.}
\email{jgrivaux@math.cnrs.fr}
\title{Parabolic automorphisms of projective surfaces \\ (after M. H. Gizatullin)}

\thanks{This research was partially supported by ANR Grant "BirPol"  ANR-11-JS01-004-01.}

\begin{document}

\maketitle

\begin{abstract}
In 1980, Gizatullin classified rational surfaces endowed with an automorphism whose action on the Neron-Severi group is parabolic: these surfaces are endowed with an elliptic fibration invariant by the automorphism. The aim of this expository  paper is to present for non-experts the details of Gizatullin's original proof, and to provide an introduction to a recent paper by Cantat and Dolgachev.
\end{abstract}

\section{Introduction}
Let $X$ be a projective complex surface. The Neron-Severi group $\mathrm{NS}\,(X)$ is a free abelian group endowed with an intersection form whose extension to $\mathrm{NS}_{\R}(X)$ has signature $(1, \mathrm{h}^{1,1}(X)-1)$.  Any automorphism of $f$ acts by pullback on $\mathrm{NS}\,(X)$, and this action is isometric. The corresponding isometry $f^*$ can be of three different types: elliptic, parabolic or hyperbolic. These situations can be read on the growth of the iterates of $f^*$. If $|| \, . \, ||$ is any norm on $\mathrm{NS}_{\R}(X)$, they correspond respectively to the following situations: $||(f^*)^n||$ is bounded, $||(f^*)^n|| \sim C n^2$ and $||(f^*)^n|| \sim \lambda^n$ for $\lambda >1$. 
This paper is concerned with the study of parabolic automorphisms of projective complex surfaces. The initial motivation to their study was that parabolic automorphisms don't come from $\mathrm{PGL}(N, \C)$ via some projective embedding $X \hookrightarrow \P^N$. Indeed, if $f$ is an automorphism coming from $\mathrm{PGL}(N, \C)$, then $f^*$ must preserve an ample class in $\mathrm{NS}\,(X)$, so $f^*$ is elliptic. The first known example of such a pair $(X, f)$, due to initially to Coble \cite{Coble} and popularised by Shafarevich, goes as follows: consider a generic pencil of cubic curves in $\mathbb{P}^2$, it has $9$ base points. Besides, all the curves in the pencil are smooth elliptic curves except $12$ nodal curves. After blowing up the nine base points, we get a elliptic surface $X$ with $12$ singular fibers and $9$ sections $s_1, \ldots, s_9$ corresponding to the exceptional divisors, called a Halphen surface (of index $1$). The section $s_1$ specifies an origin on each smooth fiber of $X$. For $2 \leq i \leq 8 $, we have a natural automorphism $\sigma_i$ of the generic fiber of $X$ given by the formula $\sigma_i(x)=x+s_i-s_1$. It is possible to prove that the $\sigma_i$'s extend to automorphisms of $X$ and generate a free abelian group of rank $8$ in $\mathrm{Aut}\,(X)$. In particular, any nonzero element in this group is parabolic since the group of automorphisms of an elliptic curve fixing the class of an ample divisor is finite.
\par \medskip
In many aspects, thisexample is a faithful illustration of parabolic automorphisms on projective surfaces. A complete classification of pairs $(X, f)$ where $f$ is a parabolic automorphism of $X$ is given in \cite{GIZ}. 
In his paper, Gizatullin considers not only parabolic automorphisms, but more generally groups of automorphisms containing only parabolic or elliptic\footnote{Gizatullin considers only parabolic elements, but most of his arguments apply to the case of groups containing elliptic elements as well as soon an they contain at least \textit{one} parabolic element.} elements. We call such groups of moderate growth, since the image of any element of the group in $\mathrm{GL}(\mathrm{NS}(X))$ has polynomial growth. Gizatullin's main result runs as follows:

\begin{theorem}[\cite{GIZ}] \label{Main}
Let $X$ be a smooth projective complex surface and $G$ be an infinite subgroup of $\mathrm{Aut}\, (X)$ of moderate growth. Then there exists a unique elliptic $G$-invariant fibration on $X$.
\end{theorem}
Of course, if $X$ admits one parabolic automorphism $f$, we can apply this theorem with the group $G=\Z $, and we get a unique $f$-invariant elliptic fibration on $X$. It turns out that it is possible to reduce Theorem \ref{Main} to the case $G=\Z $ by abstract arguments of linear algebra.
\par \medskip
In all cases except rational surfaces, parabolic automorphisms come from minimal models, and are therefore quite easy to understand. The main difficulty occurs in the case of rational surfaces. As a corollary of the classification of relatively minimal elliptic surfaces, the relative minimal model of a rational elliptic surface is a Halphen surface of some index $m$. Such surfaces are obtained by blowing up the base points of a pencil of curves of degree $3m$ in $\mathbb{P}^2$. By definition, $X$ is a Halphen surface of index $m$ if the divisor $-mK_X$ has no fixed part and $|-mK_X|$ is a pencil without base point giving the elliptic fibration.
\begin{theorem}[\cite{GIZ}] \label{Second}
Let $X$ be a Halphen surface of index $m$, $S_1, \ldots, S_{\lambda}$ the reducible fibers and $\mu_i$ the number of reducible components of $S_i$, and $s=\sum_{i=1}^{\lambda} \{\mu_i-1\}$. Then $s \leq 8$, and there exists a free abelian group $G_X$ of rank $s-8$ in $\mathrm{Aut}\,(X)$ such that every nonzero element of this group is parabolic and acts by translation along the fibers.  If $\lambda \geq 3$, $G$ has finite index in $\mathrm{Aut}\,(X)$.
\end{theorem}
The number $\lambda$ of reducible fibers is at least two, and the case $\lambda=2$ is very special since all smooth fibers of $X$ are isomorphic to a fixed elliptic curve. Such elliptic surfaces $X$ are now called Gizatullin surfaces, their automorphism group is an extension of $\C^{\times}$ by a finite group, $s=8$, and the image of the representation $\rho \colon \mathrm{Aut}\,(X) \rightarrow \mathrm{GL}\, (\mathrm{NS}\,(X))$ is finite.
\par \medskip
Let us now present applications of Gizatullin's construction. The first application lies in the theory of classification of birational maps of surfaces, which is an important subject both in complex dynamics and in algebraic geometry. One foundational result in the subject is Diller-Favre's classification theorem \cite{DF}, which we recall now. If $X$ is a projective complex surface and $f$ is a birational map of $X$, then $f$ acts on the Neron-Severi group $\mathrm{NS}\,(X)$. The conjugacy types of birational maps can be classified in four different types, which can be detected by looking at the growth of the endomorphisms $(f^*)^n$. The first type corresponds to birational maps $f$ such that $|| (f^*)^n || \sim \alpha n$. These maps are never conjugate to automorphisms of birational models on $X$ and they preserve a rational fibration. The three other remaining cases are $|| (f^*)^n ||$ bounded, $|| (f^*)^n || \sim Cn^2$ and $|| (f^*)^n || \sim C \lambda^n$. In the first two cases, Diller and Favre prove that $f$ is conjugate to an automorphism of a birational model of $X$. The reader can keep in mind the similarity between the last three cases and Nielsen-Thurston's classification of elements in the mapping class group into three types: periodic, reducible and pseudo-Anosov. The first class is now well understood (see \cite{BD2}), and constructing automorphisms in the last class is a difficult problem (see \cite{BKv}, \cite{MM} for a systematic construction of examples in this category, as well as \cite{BK}, \cite{BD} and \cite{DG} for more recent results). The second class fits exactly to Gizatullin's result: using it, we get that $f$ preserves an elliptic fibration. 
\par \medskip
One other feature of Gizatullin's theorem is to give a method to construct hyperbolic automorphisms on surfaces. This seems to be paradoxal since Gizatullin's result only deals with parabolic automorphisms. However, the key idea is the following: if $f$ and $g$ are two parabolic (or even elliptic) automorphisms of a surface generating a group $G$ of moderate growth, then $f^*$ and $g^*$ share a common nef class in $\mathrm{NS}\,(X)$, which is the class of any fiber of the $G$-invariant elliptic fibration. Therefore, if $f$ and $g$ don't share a fixed nef class in $\mathrm{NS}\, (X)$, some element in the group $G$ must be hyperbolic.
\par \medskip
Let us describe the organization of the paper. \S \ref{3} is devoted to the theory of abstract isometries of quadratic forms of signature $(1, n-1)$ on $\R^n$. In \S \ref{3.1}, we recall their standard classification into three types (elliptic, parabolic and hyperbolic). The next section (\S \ref{3.2}) is devoted to the study of special parabolic isometries, called parabolic translations. They depend on an isotropic vector $\theta$, the direction of the translation, and form an abelian group $\mathcal{T}_{\theta}$. We prove in Proposition \ref{ray} and Corollary \ref{wazomba} one of Gizatullin's main technical lemmas: if $u$ and $v$ are two parabolic translations in different directions, then $uv$ or $u^{-1}v$ must be hyperbolic. Building on this result, we prove in \S \ref{3.3} a general structure theorem (Theorem \ref{ptfixe}) for groups of isometries fixing a lattice and containing no hyperbolic elements. In \S \ref{4}, we recall classical material in birational geometry of surfaces which can be found at different places of \cite{DF}. 
In particular, we translate the problem of the existence of an $f$-invariant elliptic fibration in terms of the invariant nef class $\theta$ (Proposition \ref{nefnef}), and we also prove using the fixed point theorem of \S \ref{3.3} that it is enough to deal with  the case $G=\Z f$ in Theorem \ref{Main}. Then we settle this theorem for all surfaces except rational ones. In \S \ref{5} and \S \ref{6}, we prove Gizatullin's theorem. 
\par \medskip
Roughly speaking, the strategy goes as follows: the invariant nef class $\theta$ is always effective, we represent it by a divisor $C$. This divisor behaves exactly as a fiber of a minimal elliptic surface, we prove this in Lemmas \ref{base} and \ref{genre}. The conormal bundle $N^*_{C/X}$ has degree zero on each component of $C$, but is not always a torsion point in $\mathrm{Pic}\, (C)$. If it is a torsion point, it is easy to produce the elliptic fibration by a Riemann-Roch type argument. If not, we consider the trace morphism $\mathfrak{tr} \colon \mathrm{Pic}\, (X) \rightarrow \mathrm{Pic}\, (C)$ and prove in Proposition \ref{torsion} that $f$ acts finitely on $\mathrm{ker}\, (\mathfrak{tr})$. In Proposition \ref{elliptic}, we prove that $f$ also acts finitely on a large part of $\mathrm{im}\, (\mathfrak{tr})$. By a succession of clever tricks, it is possible from there to prove that $f$ acts finitely on $\mathrm{Pic}\, (X)$; this is done in Proposition \ref{chic}. 
\par \medskip
In \S 6 we recall the classification theory of relatively minimal rational elliptic surfaces; we prove in Proposition \ref{primitive} that they are Halphen surfaces. In Proposition \ref{sept} and Corollary \ref{sympa}, we prove a part of Theorem \ref{Second}: the existence of parabolic automorphisms imposes a constraint on the number of reducible components of the fibration, namely $s \leq 7$. We give different characterisations of Gizatullin surfaces (that is minimal elliptic rational surfaces with two singular fibers) in Proposition \ref{waza}. Then we prove the converse implication of Theorem \ref{Second}: the numerical constraint $s \leq 7$ is sufficient to guarantee the existence of parabolic automorphisms. Lastly, we characterize  minimal elliptic surfaces carrying no parabolic automorphisms in Proposition \ref{hapff}: the generic fiber must have a finite group of automorphisms over the function field $\mathbb{C}(t)$. At the end of the paper, we carry out the explicit calculation of the representation of $\mathrm{Aut}\, (X)$ on $\mathrm{NS}\, (X)$ for unnodal Halphen surfaces (that is Halphen surfaces with irreducible fibers) in Theorem \ref{classieux}. These surfaces are of crucial interest since their automorphism group is of maximal size in some sense, see \cite{CD} for a precise statement. 
\par \medskip
Throughout the paper, we work over the field of complex numbers. However, Gizatullin's arguments can be extended to any field of any characteristic with minor changes. We refer to the paper \cite{CD} for more details.
\par \bigskip
\textbf{Acknowledgements} I would like to thank Charles Favre for pointing to me Gizatullin's paper and encouraging me to write this survey, as well as Jeremy Blanc, Julie D\'eserti and Igor Dolgachev for very useful comments.

\tableofcontents

\section{Notations and conventions} \label{2}
Throughout the paper, $X$ denotes a smooth complex projective surface, which will always assumed to be rational except in \S \ref{4}. 
\par \medskip
By divisor, we will always mean $\Z$-divisor. A divisor $D=\sum_i a_i\, D_i$ on $X$ is called primitive if $\mathrm{gcd}(a_i)=1$. 
\par \medskip
If $D$ and $D'$ are two divisors on $X$, we write $D \sim D'$ (resp. $D \equiv D'$) if $D$ and $D'$ are linearly (resp. numerically) equivalent. 
\par \medskip
For any divisor $D$, we denote by $|D|$ the complete linear system of $D$, that is the set of effective divisors linearly equivalent to $D$; it is isomorphic to $\mathbb{P}\, \bigl( \mathrm{H}^0(X, \oo_X(D) \bigr)$. 
\par \medskip
The group of divisors modulo numerical equivalence is the Neron-Severi group of $X$, we denote it by $\mathrm{NS} (X)$. By Lefschetz's theorem on $(1, 1)$-classes, $\mathrm{NS}\,(X)$ is the set of Hodge classes of weight $2$ modulo torsion, this is a $\Z$-module of finite rank. We also put $\mathrm{NS}\,(X)_{\R}=\mathrm{NS}\,(X) \otimes_{\Z} \R$.
\par \medskip
If $f$ is a biregular automorphism of $X$, we denote by $f^*$ the induced action on $\mathrm{NS}\,(X)$. We will always assume that $f$ is \textit{parabolic}, which means that the induced action $f^*$ of $f$ on $\mathrm{NS}_{\R}(X)$ is parabolic. 
\par \medskip
The first Chern class map is a surjective group morphism $\mathrm{Pic}\,(X) \xrightarrow{\mathrm{c}_1} \mathrm{NS}\,(X)$, where $\mathrm{Pic}\,(X)$ is the Picard group of $X$. This morphism is an isomorphism if $X$ is a rational surface, and $\mathrm{NS}\,(X)$ is isomorphic to $\Z^r$ with $r=\chi(X)-2$.
\par \medskip
If $r$ is the rank of $\mathrm{NS}\,(X)$, the intersection pairing induces a non-degenerate bilinear form of signature $(1, r-1)$ on $X$ by the Hodge index theorem. Thus, all vector spaces included in the isotropic cone of the intersection form are lines.
\par \medskip
If $D$ is a divisor on $X$, $D$ is called a nef divisor if for any algebraic curve $C$ on $X$, $D.C \geq 0$. The same definition holds for classes in $\mathrm{NS}\,(X)_{\R}$. By Nakai-Moishezon's criterion, a nef divisor has nonnegative self-intersection.

\section{Isometries of a Lorentzian form} \label{3}

\subsection{Classification} \label{3.1}
Let $V$ be a real vector space of dimension $n$ endowed with a symmetric bilinear form of signature $(1, n-1)$. The set of nonzero elements $x$ such that $x^2 \geq 0$ has two connected components. We fix one of this connected component and denote it by $\mathfrak{N}$. 
\par \medskip
In general, an isometry maps $\mathfrak{N}$ either to $\mathfrak{N}$, either to $- \mathfrak{N}$. The index-two subgroup $\mathrm{O}_+ (V)$ of $\mathrm{O}(V)$ is the subgroup of isometries leaving $\mathfrak{N}$ invariant. 
\par \medskip
There is a complete classification of elements in $\mathrm{O}_+ (V)$. For nice pictures corresponding to these three situations, we refer the reader to Cantat's article in \cite{Milnor}.
\begin{proposition} \label{classification}
Let $u$ be in $\mathrm{O}_+ (V)$. Then three distinct situations can appear: 
\par \medskip
\begin{enumerate}
\item \textbf{u is hyperbolic}
\par
\noindent There exists $\lambda>1$ and two distinct vectors $\theta_{+}$ and $\theta_-$ in $\mathfrak{N}$ such that $u(\theta_+)=\lambda \, \theta_+$ and $u(\theta_-)=\lambda^{-1} \theta_-$. All other eigenvalues of $u$ are of modulus $1$, and $u$ is semi-simple.
\item \textbf{u is elliptic} 
\par
\noindent All eigenvalues of $u$ are of modulus $1$ and $u$ is semi-simple. Then $u$ has a fixed vector in the interior of $\mathfrak{N}$. 
\item \textbf{u is parabolic}
\par \noindent All eigenvalues of $u$ are of modulus $1$ and $u$ fixes pointwise a unique ray in $\mathfrak{N}$, which lies in the isotropic cone. Then $u$ is not semi-simple and has a unique non-trivial Jordan block which is of the form $\begin{pmatrix} 1&1&0\\ 0&1&1\\ 0&0&1 \end{pmatrix}$ where the first vector of the block directs the unique invariant isotropic ray in $\mathfrak{N}$.
\end{enumerate}
\end{proposition}

\begin{proof}
The existence of an eigenvector in $\mathfrak{N}$ follows from Brouwer's fixed point theorem applied to the set of positive half-lines in $\mathfrak{N}$, which is homeomorphic to a closed euclidian ball in $\mathbb{R}^{n-1}$. Let $\theta$ be such a vector and $\lambda$ be the corresponding eigenvalue.
\par \medskip
$*$ If $\theta$ lies in the interior of $\mathfrak{N}$, then $V=\R\, \theta \oplus {\theta}^{\perp}$. Since the bilinear form is negative definite on ${\theta}^{\perp}$, $u$ is elliptic.
\par \medskip
$*$ If $\theta$ is isotropic and $\lambda \neq 1$, then $\mathrm{im}\, (u-\lambda^{-1} \mathrm{id}) \subset \theta^{\perp}$
so that $\lambda^{-1}$ is also an eigenvalue of $u$. Hence we get two isotropic eigenvectors $\theta_+$ and $\theta_-$ corresponding to the eigenvalues $\lambda$ and $\lambda^{-1}$. Then $u$ induces an isometry of $\theta_+^{\perp} \cap \theta_-^{\perp}$, and $u$ is hyperbolic.
\par \medskip
$*$ If $\theta$ is isotropic and $\lambda=1$, and if no eigenvector of $u$ lies in the interior of $\mathfrak{N}$, we put $v=u-\textrm{id}$. If $\theta'$ is a vector in $\mathrm{ker} \, (v)$ outside $\theta^{\perp}$, then $\theta' + t \theta$ lies in the interior of $\mathfrak{N}$ for large values of $t$ and is fixed by $u$, which is impossible. Therefore $\mathrm{ker}\, (v) \subset \theta^{\perp}$. In particular, we see that $\mathbb{R \theta}$ is the unique $u$-invariant isotropic ray.
\par \medskip
Since $\theta$ is isotropic, the bilinear form is well-defined and negative definite on $\theta^{\perp}/{\R \theta}$, so that $u$ induces a semi-simple endomorphism $\overline{u}$ on $\theta^{\perp}/{\R \theta}$. Let $P$ be the minimal polynomial of $\overline{u}$, $P$ has simple complex roots. Then there exists a linear form $\ell$ on $\theta^{\perp}$ such that for any $x$ orthogonal to $\theta$, $P(u)(x)=\ell(x) \, \theta$. Let $E$ be the kernel of $\ell$. Remark that
\[
\ell(x)\, \theta= u\{\ell(x)\, \theta \}=u \,\{P(u)(x)\}=P(u)(u(x))=\ell(u(x))\, \theta
\]
so that $\ell \circ u=\ell$, which implies that $E$ is stable by $u$. Since $P(u_{|E})=0$, $u_{|E}$ is semi-simple. 
\par \medskip
Assume that $\theta$ doesn't belong to $E$. Then the quadratic form is negative definite on $E$, and $V=E \oplus E^{\perp}$. On $E^{\perp}$, the quadratic form has signature $(1,1)$. Then the situation becomes easy, because the isotropic cone consists of two lines, which are either preserved or swapped. If they are preserved, we get the identity map. If they are swapped, we get a reflexion along a line in the interior of the isotropic cone, hence an elliptic element. In all cases we get a contradiction.
\par \medskip
Assume that $u_{| \theta^{\perp}}$ is semi-simple. Since $\mathrm{ker}\, (v) \subset \theta^{\perp}$, we can write $\theta^{\perp}=\mathrm{ker}\, (v)\, \oplus W$ where $W$ is stable by $v$ and $v_{|W}$ is an isomorphism. Now $\mathrm{im}\,(v) =\mathrm{ker}\,(v)^{\perp}$, and it follows that $\mathrm{im}\, (v)=\mathbb{R} \theta \oplus W$. Let $\zeta$ be such that $v(\zeta)=\theta$. Then $u(\zeta)=\zeta+\theta$, so that $u(\zeta)^2=\zeta^2+2 (\zeta. \theta)$. It follows that $\zeta. \theta=0$, and we get a contradiction. In particular $\ell$ is nonzero.  
\par \medskip
Let $F$ be the orthogonal of the subspace $E$, it is a plane in $V$ stable by $u$, containing $\theta$ and contained in $\theta^{\perp}$. Let $\theta'$ be a vector in $F$ such that $\{\theta, \theta'\}$ is a basis of $F$ and write $u(\theta')=\alpha \theta + \beta \theta'$. Since $\theta$ and $\theta'$ are linearly independent, $\theta'^2 <0$. Besides, $u(\theta')^2=\theta'^2$ so that $\beta^2=1$. Assume that $\beta=-1$. If $x=\theta'-\frac{\alpha}{2} \theta$, then $u(x)=-x$, so that $u_{\theta^{\perp}}$ is semi-simple. Thus $\beta=1$. Since $\alpha \neq 0$ we can also assume that $\alpha=1$.
\par \medskip
Let $v=u-\textrm{id}$. We claim that $\mathrm{ker}\,(v) \subset E$. Indeed, if $u(x)=x$, we know that $x\in \theta^{\perp}$. If $x \notin E$, then $P(u)(x) \neq 0$. But $P(u)(x)=P(1) \, x$ and since $\theta \in E$, $P(1)=0$ and we get a contradiction. This proves the claim.
\par \medskip
Since $\mathrm{im}\,(v) \subseteq \mathrm{ker}\,(v)^{\perp}$, $\mathrm{im}\,(v)$ contains $F$. Let $\theta''$ be such that $v(\theta'')=\theta'$. Since $v(\theta^{\perp}) \subset E$, $\theta'' \notin \theta^{\perp}$. The subspace generated with $\theta$, $\theta'$ and $\theta''$ is a $3 \times 3$ Jordan block for $u$. 
\end{proof}

\begin{remark} \label{bof}
Elements of the group $\mathrm{O}_{+}(V)$ can be distinguished by the growth of the norm of their iterates. More precisely:
\begin{enumerate}
\item[--] If $u$ is hyperbolic, $||u^n|| \sim C\lambda^{n}$.
\par \smallskip
\item[--] If $u$ is elliptic, $||u^n||$ is bounded.
\par \smallskip
\item[--] If $u$ is parabolic, $||u^n|| \sim C {n}^2$.
\end{enumerate}
\end{remark}
We can sum up the two main properties of parabolic isometries which will be used in the sequel:

\begin{lemma} \label{lemmenef}
Let $u$ be a parabolic element of $\mathrm{O}_{+}(V)$ and $\theta$ be an isotropic fixed vector of $u$.
\begin{enumerate}
\item If $\alpha$ is an eigenvector of $u$, $\alpha^2 \leq 0$.
\par \smallskip
\item If $\alpha$ is fixed by $u$, then $\alpha \, . \,\theta=0$. Besides, if $\alpha^2=0$, $\alpha$ and $\theta$ are proportional. 
\end{enumerate}
\end{lemma}
\subsection{Parabolic isometries} \label{3.2}
The elements which are the most difficult to understand in $\mathrm{O}_{+}(V)$ are parabolic ones. In this section, we consider a distinguished subset of parabolic elements associated with any isotropic vector.

\par \medskip
Let $\theta$ be an isotropic vector in $\mathfrak{N}$ and $Q_{\theta}=\theta^{\perp}/ \R \theta$. The quadratic form is negative definite on $Q_{\theta}$. Indeed, if $x\, . \, \theta=0$, $x^2 \leq 0$ with equality if and only if $x$ and $\theta$ are proportional, so that $x=0$ in $Q_{\theta}$. If 
\[
\mathrm{O}_{+}(V)_{\theta}=\{ u \in \mathrm{O}_+(V) \, \, \textrm{such that} \, \, u(\theta)=\theta\}
\]
we have a natural group morphism 
\[
\chi_{\theta} \colon \mathrm{O}_{+}(V)_{\theta} \rightarrow \mathrm{O}(Q_{\theta}),
\] 
and we denote by $\mathcal{T}_{\theta}$ its kernel.  Let us fix another isotropic vector $\eta$ in $\mathfrak{N}$ which is not collinear to $\theta$, and let $\pi \colon V \rightarrow \theta^{\perp} \cap \eta^{\perp}$ be the orthogonal projection along the plane generated by $\theta$ and $\eta$.
\begin{proposition} \label{commutatif}$ $
\par
\begin{enumerate} 
\item The map $\varphi \colon \mathcal{T}_{\theta} \rightarrow \theta^{\perp} \cap \eta^{\perp}$ given by $\varphi(u)=\pi \{ u(\eta) \}$ is a group isomorphism.
\item Any element in $\mathcal{T}_{\theta} \setminus \{ \textrm{id} \}$ is parabolic. 
\end{enumerate}
\end{proposition}

\begin{proof}
We have $V=\{\theta^{\perp} \cap \eta^{\perp} \oplus \R \theta\} \oplus \R \eta=\theta^{\perp} \oplus \R \eta$. Let $u$ be in $G_{\theta}$, and denote by $\zeta$ the element $\varphi(u)$. Let us decompose $u(\eta)$ as $a \theta + b \eta + \zeta$. Then $0=u(\eta)^2=2ab\, (\theta\, . \, \eta)+ \zeta^2$ and we get 
\[
ab=-\dfrac{\zeta^2}{2\, (\theta. \eta)} \cdot 
\]
Since $u(\theta)=\theta, 
\theta\, .\, \eta=\theta \, u(\eta)=b\, (\theta \, . \, \eta)$ 
so that $b=1$. This gives \[a=-\dfrac{\zeta^2}{2\, (\theta. \eta)} \cdot \]
By hypothesis, there exists a linear form $\lambda \colon \theta^{\perp}\cap \eta^{\perp} \rightarrow \R$ such that for any $x$ in $\theta^{\perp} \cap \eta^{\perp}$, $u(x)=x+\lambda(x) \, \theta$.
Then we have
\[
0=x\, . \, \eta = u(x)\,.\, u(\eta)=x \, . \, \zeta+ \lambda(x) \, \theta\, . \, \eta
\]
so that 
\[
\lambda(x)=-\dfrac{(x\, . \, \zeta)}{(\theta\, . \, \eta)} \cdot
\] 
This proves that $u$ can be reconstructed from $\zeta$. For any $\zeta$ in $\theta^{\perp} \cap \eta^{\perp}$, we can define a map $u_{\zeta}$ fixing $\theta$ by the above formul\ae, and it is an isometry. This proves that $\varphi$ is a bijection.
To prove that $\varphi$ is a morphism, let $u$ and $u'$ be in $G_{\theta}$, and put $u''=u' \circ u$. Then
\[
\qquad \zeta''=\pi \{u' (u(\eta))\}=\pi \{u' (\zeta +a \theta + \eta) \}=\pi \{ \zeta + \lambda(\zeta) \theta + a \theta + \zeta' + a' \theta + \eta\}=\zeta + \zeta'.
\]
It remains to prove that $u$ is parabolic if $\zeta \neq 0$. This is easy: if $x=\alpha \theta + \beta \eta + y$ where $y$ is in $\theta^{\perp} \cap \eta^{\perp}$, then
$u(x)=\{\alpha+ \lambda(y) \} \theta + \{\beta \zeta + y\}$. Thus, if $u(x)=x$, we have $\lambda(y)=0$ and $\beta=0$. But in this case, $x^2=y^2 \leq 0$ with equality if and only if $y=0$. It follows that $\R_{+}\theta$ is the only fixed ray in $\mathfrak{N}$, so that $u$ is parabolic.
\end{proof}
\begin{definition}
Nonzero elements in $\mathcal{T}_{\theta}$ are called parabolic translations along $\theta$.
\end{definition}
This definition is justified by the fact that elements in the group $\mathcal{T}_{\theta}$ act by translation in the direction $\theta$ on $\theta^{\perp}$.
\begin{proposition} \label{ray}
Let $\theta$, $\eta$ be two isotropic and non-collinear vectors in $\mathfrak{N}$, and $\varphi \colon \mathcal{T}_{\theta} \rightarrow \theta^{\perp} \cap \eta^{\perp}$ and $\psi \colon \mathcal{T}_{\eta} \rightarrow \theta^{\perp} \cap \eta^{\perp}$ the corresponding isomorphisms. Let $u$ and $v$ be respective nonzero elements of $\mathcal{T}_{\theta}$ and $\mathcal{T}_{\eta}$, and assume that there exists an element $x$ in $\mathfrak{N}$ such that $u(x)=v(x)$. Then there exists $t > 0$ such that $\psi(v)=t\, \varphi(u)$. 
\end{proposition}
\begin{proof}
Let us write $x$ as $\alpha \theta + \beta \eta + y$ where $y$ is in $\theta^{\perp} \cap \eta^{\perp}$. Then
\[
\qquad u(x)=\alpha \, \theta + \beta \zeta + y + \lambda(y) \, \theta \quad \textrm{and} \quad v(x)= \alpha \, \zeta' + \beta \eta + y + \mu(y) \, \eta.
\]
\noindent Therefore, if $u(x)=v(x)$, 
\[
\qquad \{\alpha + \lambda(y)\} \, \theta - \{\beta + \mu(y) \}\, \eta + \{\beta \zeta - \alpha \zeta' \} =0
\] 
Hence $\beta \zeta - \alpha \zeta' =0$. We claim that $x$ doesn't belong to the two rays $\R \theta$ and $\R \eta$. Indeed, if $y=0$,  $\alpha=\beta=0$ so that $u(x)=0$. Thus, since $x$ lies in $\mathfrak{N}$, $x \, . \, \theta >0$ and $x \, . \, \eta >0$ so that $\alpha >0$ and $\beta >0$. Hence $\zeta'=\dfrac{\beta}{\alpha}\, \zeta$ and $\dfrac{\beta}{\alpha}>0$.
\end{proof}
\begin{corollary} \label{wazomba}
Let $\theta$, $\eta$ two isotropic and non-collinear vectors in $\mathfrak{N}$ and $u$ and $v$ be respective nonzero elements of $\mathcal{T}_{\theta}$ and $\mathcal{T}_{\eta}$. Then $u^{-1}v$ or $uv$ is hyperbolic.
\end{corollary}
\begin{proof}
If $u^{-1}v$ is not hyperbolic, then there exists a nonzero vector $x$ in $\mathfrak{N}$ fixed by $u^{-1} v$. Thus, thanks to Proposition \ref{ray}, there exists $t>0$ such that $\psi(v)=t\, \varphi(u)$. By the same argument, if $uv$ is not hyperbolic, there exists $s>0$ such that $\psi(v)=s\, \varphi(u^{-1})=-s\, \varphi(u)$. This gives a contradiction.
\end{proof}
\subsection{A fixed point theorem} \label{3.3}
In this section, we fix a lattice $\Lambda$ of rank $n$ in $V$ and assume that the bilinear form on $V$ takes integral values on the lattice $\Lambda$. We denote by $\mathrm{O}_{+}(\Lambda)$ the subgroup of $\mathrm{O}_{+}(V)$ fixing the lattice.
We start by a simple characterisation of elliptic isometries fixing $\Lambda$:
\begin{lemma} \label{fini} $ $ \par
\begin{enumerate}
\item An element of $\mathrm{O}_{+}(\Lambda)$ is elliptic if and only if it is of finite order.
\par \smallskip
\item An element $u$ of $\mathrm{O}_{+}(\Lambda)$ is parabolic if and only if it is quasi-unipotent (which means that there exists an integer $k$ such that $u^k-1$ is a nonzero nilpotent element) and of infinite order.
\end{enumerate}
\end{lemma}
\begin{proof} $ $ \par
\begin{enumerate}
\item
A finite element is obviously elliptic. Conversely, if $u$ is an elliptic element of $\mathrm{O}_{+}(\Lambda)$, there exists a fixed vector $\alpha$ in the interior of $\mathfrak{N}$. Since $\ker\, (u-\mathrm{id})$ is defined over $\Q$, we can find such an $\alpha$ defined over $\Q$. In that case, $u$ must act finitely on $\alpha^{\perp} \cap \Lambda$ and we are done.
\par \smallskip
\item A quasi-unipotent element which is of infinite order is parabolic (since it is not semi-simple). Conversely, if $g$ is a parabolic element in $\mathrm{O}_{+}(\Lambda)$, the characteristic polynomial of $g$ has rational coefficients and all its roots are of modulus one. Therefore all eigenvalues of $g$ are roots of unity thanks to Kronecker's theorem.
\end{enumerate}
\end{proof}

One of the most important properties of parabolic isometries fixing $\Lambda$ is the following:
\begin{proposition} \label{gauss}
Let $u$ be a parabolic element in $\mathrm{O}_{+}(\Lambda)$. Then \emph{:}
\begin{enumerate}
\item There exists a vector $\theta$ in $\mathfrak{N} \cap \Lambda$ such that $u(\theta)=\theta$.
\item There exists $k>0$ such that $u^k$ belongs to $\mathcal{T}_{\theta}$.
\end{enumerate}
\end{proposition}
\begin{proof}$ $ \par
\begin{enumerate}
\item Let $W= \mathrm{ker}\,(f-\mathrm{id})$, and assume that the line $\R \theta$ doesn't meet ${\Lambda}_{\Q}$. 
Then the quadratic form $q$ is negative definite on $\theta^{\perp} \cap W_{\Q}$. We can decompose $q_{W_{\Q}}$ as $-\sum_i \ell_i^2$ where the $\ell_i$'s are linear forms on $W_{\Q}$. Then $q$ is also negative definite on $W$, but $q(\theta)=0$ so we get a contradiction.
\par \smallskip
\item By the first point, we know that we can choose an isotropic invariant vector $\theta$ in $\Lambda$. Let us consider the free abelian group $\Sigma:=(\theta^{\perp} \cap \Lambda)/ \Z \theta$, the induced quadratic form is negative definite. Therefore, since $u$ is an isometry, the action of $u$ is finite on $\Sigma$, so that an iterate of $u$ belongs to $\mathcal{T}_{\theta}$.

\end{enumerate}
\end{proof}

The definition below is motivated by Remark \ref{bof}.
\begin{definition}
A subgroup $G$ of $\mathrm{O}_+ (V)$ is of \textit{moderate growth} if it contains no hyperbolic element.
\end{definition}
Among groups of moderate growth, the most simple ones are finite subgroups of $\mathrm{O}_+ (V)$. Recall the following well-known fact:
\begin{lemma} \label{burnside}
Any torsion subgroup of $\mathrm{GL}(n, \Q)$ is finite.
\end{lemma}

\begin{proof}
Let $g$ be an element in $G$, and $\zeta$ be an eigenvalue of $g$. If $m$ is the smallest positive integer such that $\zeta^m=1$, 
then $\varphi(m)=\mathrm{deg}_{\Q} (\zeta) \leq n$ where $\varphi(m)=\sum_{d |m} d $. Since $\varphi(k)\underset{k \rightarrow + \infty}{\longrightarrow} + \infty$, there are finitely many possibilities for $m$. Therefore, there exists a constant $c(n)$ such that the order of any $g$ in $G$ divides $c(n)$. This means that $G$ has finite exponent in $\mathrm{GL}(n, \C)$, and the Lemma follows from Burnside's theorem.
\end{proof}
As a consequence of Lemmas \ref{fini} and \ref{burnside}, we get:

\begin{corollary} \label{burnside}
A subgroup of $\mathrm{O}_+ (\Lambda)$ is finite if and only if all its elements are elliptic.
\end{corollary}
We now concentrate on infinite groups of moderate growth. The main theorem we want to prove is Gizatullin's fixed point theorem:
\begin{theorem} \label{ptfixe}
Let $G$ be an infinite subgroup of moderate growth in $\mathrm{O}_+ (\Lambda)$. Then \emph{:}
\begin{enumerate}
\item There exists an isotropic element $\theta$ in $\mathfrak{N} \cap \Lambda$ such that for any element $g$ in $G$, $g(\theta)=\theta$.  
\par \smallskip
\item The group $G$ can be written as $G=\Z^{r} \rtimes H$ where $H$ is a finite group and $r>0$.
\end{enumerate}

\end{theorem} 
\begin{proof} $ $ \par
\begin{enumerate}
\item
Thanks to Corollary \ref{burnside}, $G$ contains parabolic elements. Let $g$ be a parabolic element in $G$ and $\theta$ be an isotropic vector. Let $\Lambda^*=(\theta^{\perp} \cap \Lambda)/ \Z \theta$. Since the induced quadratic form on $\Lambda^*$ is negative definite, and an iterate of $g$ acts finitely on $\Lambda^*$; hence $g^k$ is in $\mathcal{T}_{\theta}$ for some integer $k$. 
\par \medskip
\noindent Let $\tilde{g}$ be another element of $G$, and assume that $\tilde{g}$ doesn't fix $\theta$. We put $\eta=\tilde{g}(\theta)$. If $u={g}^k$ and $v=\tilde{g} g^k \tilde{g}^{-1}$, then $u$ and $v$ are nonzero elements of $\mathcal{T}_{\theta}$ and $\mathcal{T}_{\eta}$ respectively. Thanks to Corollary \ref{waza}, $G$ contains hyperbolic elements, which is impossible since it is of moderate growth.
\par \medskip
\item Let us consider the natural group morphism 
\[
\varepsilon \colon G \hookrightarrow \mathrm{O}_{+}(V) \rightarrow \mathrm{O}(\Lambda^*).
\]
The image of $\varepsilon$ being finite, $\ker \, (\varepsilon)$ is a normal subgroup of finite index in $G$. This subgroup is included in $\mathcal{T}_{\theta}$, so it is commutative. Besides, it has no torsion thanks to Proposition \ref{commutatif} (1), and is countable as a subgroup of $\mathrm{GL}_n(\Z)$. Thus it must be isomorphic to $\Z^r$ for some $r$.
\end{enumerate}

\end{proof}

\section{Background material on surfaces} \label{4}

\subsection{The invariant nef class} \label{4.1}
Let us consider a pair $(X, f)$ where $X$ is a smooth complex projective surface and $f$ is an automorphism of $X$ whose action on $\mathrm{NS}(X)_{\R}$ is a parabolic isometry. 
\begin{proposition} \label{tropmalin}
There exists a unique non-divisible nef vector $\theta$ in $\mathrm{NS}\,(X) \cap \ker \left( f^* - \mathrm{id} \right)$. Besides, $\theta$ satisfies $\theta^2=0$ and ${K}_X. \theta=0$.
\end{proposition}
\begin{proof}
Let $\mathcal{S}$ be the space of half-lines $\R_{+} \mu$ where $\mu$ runs through nef classes in $\mathrm{NS}\,(X)$. Taking a suitable affine section of the nef cone so that each half-line in $\mathcal{S}$ is given by the intersection with an affine hyperplane, we see that $\mathcal{S}$ is bounded and convex, hence homeomorphic to a closed euclidian ball in $\R^{n-1}$. By Brouwer's fixed point theorem, $f^*$ must fix a point in $\mathcal{S}$. This implies that $f^* \theta=\lambda \,\theta$ for some nef vector $\theta$ and some positive real number $\lambda$ which must be one as $f$ is parabolic. 
\par \medskip
Since $\theta$ is nef, $\theta^2 \geq 0$. By Lemma \ref{lemmenef} (1), $\theta^2=0$ and by Lemma  \ref{lemmenef} (2), $K_X . \theta=0$.
It remains to prove that $\theta$ can be chosen in $\mathrm{NS}\,(X)$. This follows from Lemma \ref{gauss} (1). Since $\R \theta$ is the unique fixed isotropic ray, $\theta$ is unique up to scaling. It is completely normalized if it is assumed to be non-divisible.
\end{proof}
\begin{proposition} \label{mg}
Let $G$ be an infinite group of automorphisms of $X$ having moderate growth. Then there exists a $G$-invariant nef class $\theta$ in $\mathrm{NS}\,(X)$. 
\end{proposition}
\begin{proof}
This follows directly from Theorem \ref{ptfixe} and Proposition \ref{tropmalin}.
\end{proof}
\subsection{Constructing elliptic fibrations} \label{4.2}
In this section, our aim is to translate the question of the existence of $f$-invariant elliptic fibrations in terms of the invariant nef class $\theta$.
\begin{proposition} \label{nefnef}
If $(X, f)$ is given, then $X$ admits an invariant elliptic fibration if and only if a multiple $N \theta$ of the $f$-invariant nef class can be lifted to a divisor $D$ in the Picard group $\mathrm{Pic}\,(X)$ such that $\mathrm{dim} \, |D| =1$. Besides, such a fibration is unique.
\end{proposition}
\begin{proof}
Let us consider a pair $(X, f)$ and assume that $X$ admits a fibration $X \xrightarrow{\pi} {C}$ invariant by $f$ whose general fiber is a smooth elliptic curve, where $C$ is a smooth algebraic curve of genus $g$. Let us denote by $\beta$ the class of a general fiber $X_z=\pi^{-1}(z)$ in $\mathrm{NS}\,(X)$. Then $f^* \beta=\beta$. The class $\beta$ is obviously nef, so that it is a multiple of $\theta$. This implies that the fibration $(\pi, C)$ is unique: if $\pi$ and $\pi'$ are two distinct $f$-invariant elliptic fibrations, then $\beta . \,\beta' >0$; but $\theta^2=0$. 
\par \medskip
Let $C \xrightarrow{\varphi} \mathbb{P}^1$ be any branched covering (we call $N$ its degree), and let us consider the composition $X\xrightarrow{\varphi \,\circ\, \pi} \mathbb{P}^1$. Let $D$ be a generic fiber of this map. It is a finite union of the fibers of $\pi$, so that the class of $D$ in $\mathrm{NS}\,(X)$ is $N \beta$. Besides, $\mathrm{dim} \, |D| \geq 1$. In fact $\mathrm{dim} \, |D|=1$, otherwise $D^2$ would be positive. This yields the first implication in the proposition.
\par \medskip
To prove the converse implication, let $N$ be a positive integer such that if $N \theta$ can be lifted to a divisor $D$ with $\mathrm{dim} \, |D|=1$. Let us decompose $D$ as $F+M$, where $F$ is the fixed part (so that $|D|=|M|$).
Then $0=D^2=D.F+D.M$ and since $D$ is nef, $D.M=0$. Since $|M|$ has no fixed component, $M^2 \geq 0$ so that the intersection pairing is semi-positive on the vector space generated by $D$ and $M$. It follows that $D$ and $M$ are proportional, so that $M$ is still a lift of a multiple of $\theta$ in $\mathrm{Pic}\,(X)$.
\par \medskip
Since $M$ has no fixed component and $M^2=0$, $|M|$ is basepoint free. By the Stein factorisation theorem, the generic fiber of the associated Kodaira map $X \rightarrow |M|^*$ is the disjoint union of smooth curves of genus $g$. The class of each of these curves in the Neron-Severi group is a multiple of $\theta$. Since $\theta^2=\theta . K_X=0$, the genus formula implies $g=1$. To conclude, we take the Stein factorisation of the Kodaira map to get a true  elliptic fibration.
\par \medskip
It remains to prove that this fibration is $f$-invariant. If $\mathcal{C}$ is a fiber of the fibration, then $f(\mathcal{C})$ is numerically equivalent to $\mathcal{C}$ (since $f^* \theta=\theta$), so that $\mathcal{C}. f(\mathcal{C})=0$. Therefore, $f(\mathcal{C})$ is another fiber of the fibration.
\end{proof}
\begin{remark} \label{malin}
The unicity of the fibration implies that any $f^N$-elliptic fibration (for a positive integer $N$) is $f$-invariant.
\end{remark}
In view of the preceding proposition, it is natural to try to produce sections of $D$ by applying the Riemann-Roch theorem. Using Serre duality, we have
\begin{equation} \label{RR}
\mathrm{h}^0(D)+\mathrm{h}^0(K_X-D) \geq \chi(\mathcal{O}_X)+\frac{1}{2} D.(D-K_X)=\chi({\mathcal{O}_X}).
\end{equation}
In the next section, we will use this inequality to solve the case where the minimal model of $X$ is a $K3$-surface.
\begin{corollary} \label{reduction}
If Theorem \ref{Main} holds for $G=\mathbb{Z}$, then it holds in the general case.
\end{corollary}

\begin{proof}
Let $G$ be an infinite subgroup of $\mathrm{Aut}\,(X)$ of moderate growth, $f$ be a parabolic element of $X$, and assume that there exists an $f$-invariant elliptic fibration $\mathcal{C}$ on $X$. If $\theta$ is the invariant nef class of $X$, then $G$ fixes $\theta$ by Proposition \ref {mg}. This proves that $\mathcal{C}$ is $G$-invariant.	 
\end{proof}

\subsection{Kodaira's classification} \label{4.3}
Let us take $(X, f)$ as before. The first natural step to classify $(X, f)$ would be to find what is the minimal model of $X$. It turns out that we can rule out some cases without difficulties. Let $\kappa(X)$ be the Kodaira dimension of $X$.
\par \medskip
-- If $\kappa(X)=2$, then $X$ is of general type so its automorphism group is finite. Therefore this case doesn't occur in our study.
\par \bigskip
-- If $\kappa(X)=1$, we can completely understand the situation by looking at the Itaka fibration $X \dasharrow |mK_X|^*$ for $m >> 0$, which is $\mathrm{Aut}\,(X)$-invariant. Let $F$ be the fixed part of $|mK_X|$ and $D=mK_X-F$. 
\begin{lemma} The linear system $|D|$ is a base point free pencil, whose generic fiber is a finite union of elliptic curves. 
\end{lemma}
\begin{proof}
If $X$ is minimal, we refer the reader to \cite[pp. 574-575]{GH}. If $X$ is not minimal, let $Z$ be its minimal model and $X \xrightarrow{\pi} Z$ the projection. Then $K_X=\pi^*K_Z+E$, where $E$ is a divisor contracted by $\pi$, so that $|mK_X|=|mK_Z|=|D|$. 
\end{proof}
We can now consider the Stein factorisation $X \rightarrow Y \rightarrow Z$ of $\pi$. In this way, we get an $\mathrm{Aut} (X)$-invariant elliptic fibration $X \rightarrow Y$.
\par \bigskip
-- If $\kappa(X)=0$, the minimal model of $X$ is either a $K3$ surface, an Enriques surface, or a bielliptic surface. We start by noticing that we can argue directly in this case on the minimal model:
\begin{lemma}
If $\kappa(X)\!=\!0$, every automorphism of $X$ is induced by an automorphism of its minimal model. 
\end{lemma}
\begin{proof}
Let $Z$ be the minimal model of $X$ and $\pi$ be the associated projection. By classification of minimal surfaces of Kodaira dimension zero, there exists a positive integer $m$ such that $mK_Z$ is trivial. Therefore, $mK_X$ is an effective divisor $\mathcal{E}$ whose support is exactly the exceptional locus of $\pi$, and $|mK_X|=\{\mathcal{E}\}$. It follows that $\mathcal{E}$ is invariant by $f$, so that $f$ descends to $Z$.
\end{proof}
\par \medskip
$*$ Let us deal with the K3 surface case.  We pick any lift $D$ of $\theta$ in $\mathrm{Pic}\,(X)$. Since $\chi(\oo_X)=2$, we get by (\ref{RR})
\[
\mathrm{h}^0(D)+\mathrm{h}^0(-D) \geq 2.
\]
Since $D$ is nef, $-D$ cannot be effective, so that $\mathrm{h}^0(-D)=0$. We conclude using Proposition \ref{nefnef}.
\par \medskip
$*$ This argument doesn't work directly for Enriques surfaces, but we can reduce to the K3 case by arguing as follows: if $X$ is an Enriques surface, its universal cover $\widetilde{X}$ is a K3 surface, and $f$ lifts to an automorphism $\tilde{f}$ of $\widetilde{X}$. Besides, $\tilde{f}$ is still parabolic. Therefore, we get an $\tilde{f}$-invariant elliptic fibration $\pi$ on $\widetilde{X}$. 
\par \medskip
If $\sigma$ is the involution on $\widetilde{X}$ such that $X=\widetilde{X}/\sigma$, then $\tilde{f}=\sigma \circ \tilde{f} \circ \sigma^{-1}$, by the unicity of the invariant fibration, $\pi \circ \sigma=\pi$. Thus, $\pi$ descends to $X$.
\par \medskip
$*$ The case of abelian surfaces is straightforward: an automorphism of the abelian surface $\C^2/ \Lambda$ is given by some  matrix $M$ in $\mathrm{GL}(2; \Lambda)$. Up to replacing $M$ by an iterate, we can assume that this matrix is unipotent. If $M= \mathrm{id} + N$, then the image of $N \colon \Lambda \rightarrow \Lambda$ is a sub-lattice $\Lambda^*$ of $\Lambda$ spanning a complex line $L$ in $\mathbb{C}^2$. Then the elliptic fibration $\mathbb{C}^2/ \Lambda \xrightarrow{N} L / \Lambda^*$ is invariant by $M$.
\par \medskip
$*$ It remains to deal with the case of bi-elliptic surfaces. But this is easy because they are already endowed with an elliptic fibration invariant by the whole automorphism group.
\par \bigskip
-- If $\kappa(X)=-\infty$, then either $X$ is a rational surface, or the minimal model of $X$ is a ruled surface over a curve of genus $g \geq 1$. The rational surface case is rather difficult, and corresponds to Gizatullin's result;
we leave it apart for the moment. For blowups of ruled surfaces, we remark that the automorphism group must preserve the ruling. Indeed, for any fiber $\mathcal{C}$, the projection of $f(\mathcal{C})$ on the base of the ruling must be constant since $\mathcal{C}$ has genus zero. Therefore, an iterate of $f$ descends to an automorphism of the minimal model $Z$. 
\par \medskip
We know that $Z$ can be written as $\mathbb{P}(E)$ where $E$ is a holomorphic rank $2$ bundle on the base of the ruling. By the Leray-Hirsh theorem, $\mathrm{H}^{1,1}(Z)$ is the plane generated by the first Chern class $\mathrm{c}_1(\mathcal{O}_E(1))$ of the relative tautological bundle and the pull-back of the fundamental class in $\mathrm{H}^{1, 1}(\mathbb{P}^1)$. Thus, $f^*$ acts by the identity on $\mathrm{H}^{1,1}(Z)$, hence on $\mathrm{H}^{1,1}(X)$.
\section{The rational surface case} \label{5}
\subsection{Statement of the result} \label{5.1}
From now on, $X$ will always be a rational surface, so that $\mathrm{h}^1(X, \mathcal{O}_X)=\mathrm{h}^2(X, \mathcal{O}_X)=0$. It follows that $\mathrm{Pic}\,(X) \simeq \mathrm{NS}\,(X) \simeq \mathrm{H}^2(X, \Z)$, which imply that numerical and linear equivalence agree. In this section, we prove the following result:
\begin{theorem}[\cite{GIZ}] \label{Gizzz}
Let $X$ be a rational surface and $f$ be a parabolic automorphism of $X$. If $\theta$ is the nef $f$-invariant class in $\mathrm{NS}\,(X)$, then there exists an integer $N$ such that $\mathrm{dim}\, |N \theta|=1$.
\end{theorem}
Thanks to Proposition \ref{mg} and Corollary \ref{reduction}, this theorem is equivalent to Theorem \ref{Main} for rational surfaces and is the most difficult result in Gizatullin's paper.
%
%
\subsection{Properties of the invariant curve} \label{5.2}
The divisor ${K}_X-\theta$ is never effective. Indeed, if $H$ is an ample divisor, $K_X. H <0$ so that $(K_X-\theta).H <0$. Therefore, we obtain by (\ref{RR}) that $|\theta| \neq \varnothing$, so that $\theta$ can be represented by a possibly non reduced and non irreducible curve $C$. We will write the curve $C$ as the divisor $\sum_{i=1}^d a_i \,C_i$ where the $C_i$ are irreducible. Since $\theta$ is non divisible in $\mathrm{NS}\,(X)$, $C$ is primitive. 
\par \medskip
In the sequel, we will make the following assumptions, and we are seeking for a contradiction:
\par \medskip
\textbf{Assumptions}
\begin{itemize}
\item[(1)] We have $|N \theta|=\{NC\}$ for all positive integers $N$.
\item[(2)] For any positive integer $k$, the pair $(X, f^k)$ is minimal.
\end{itemize}
\par \medskip
Let us say a few words on $(2)$. If for some integer $k$ the map $f^k$ descends to an automorphism $g$ of a blow-down $Y$ of $X$, then we can still argue with $(Y, g)$. The corresponding invariant nef class will satisfy $(1)$. Thanks to Remark \ref{malin}, we don't lose anything concerning the fibration when replacing $f$ by an iterate.
\par \medskip
We study thoroughly the geometry of $C$. Let us start with a simple lemma.
\begin{lemma} \label{ttsimple}
If $D_1$ and $D_2$ are two effective divisors whose classes are proportional to $\theta$, then $D_1$ and $D_2$ are proportional (as divisors).
\end{lemma}
\begin {proof} There exists integers $N$, $N_1$, and $N_2$ such that $N_1 D_1 \equiv N_2 D_2 \equiv N \theta$. Therefore, $N_1 D_1$ and $N_2 D_2$ belong to $|N \theta|$ so they are equal.
\end{proof}
The following lemma proves that $C$ looks like a fiber of a minimal elliptic surface.
\begin{lemma} \label{base}  $ $
\begin{enumerate} 
\item For $1 \leq i \leq d$, $K_X. C_i=0$ and $C.C_i=0$. If $d \geq 2$, $C_i^2<0$.
\item The classes of the components $C_i$ in $\mathrm{NS}\,(X)$ are linearly independent.
\item The intersection form is nonpositive on the $\Z$-module spanned by the $C_i$'s.
\item If $D$ is a divisor supported in $C$ such that $D^2=0$, then $D$ is a multiple of $C$.
\end{enumerate}
\end{lemma}
\begin{proof} $ $
\begin{enumerate}

\item Up to replacing $f$ by an iterate, we can assume that all the components $C_i$ of the curve $C$ are fixed by $f$. By Lemma \ref{lemmenef}, $C_i^2\leq 0$ and $C.K_X=C.C_i=0$ for all $i$. Assume that $d \geq 2$. If $C_i^2=0$, then $C$ and $C_i$ are proportional, which would imply that $C$ is divisible in $\mathrm{NS}\,(X)$. Therefore $C_i^2<0$. If $K_X.C_i<0$, then $C_i$ is a smooth and $f$-invariant exceptional rational curve. This contradicts Assumption (2). Thus $K_X.C_i \geq 0$. Since $K_X.C=0$, it follows that $K_X.C_i=0$ for all $i$.
\par \medskip
\item If there is a linear relation among the curves $C_i$, we can write it as $D_1 \equiv D_2$, where $D_1$ and $D_2$ are linear combinations of the $C_i$ with positive coefficients (hence effective divisors) having no component in common. We have $D_1^2=D_1.D_2 \geq 0$. On the other hand $C. D_1=0$ and $C^2=0$, so by the Hodge index theorem $C$ and $D_1$ are proportional. This contradicts Lemma \ref{ttsimple}.
\par \medskip
\item Any divisor $D$ in the span of the $C_i$'s is $f$-invariant, so that Lemma \ref{lemmenef} (1) yields $D^2 \leq 0$.
\par \medskip
\item If $D^2=D.C=0$, then $D$ and $C$ are numerically proportional. Therefore, there exists two integers $a$ and $b$ such that $aD-bC \equiv 0$. By Lemma \ref{ttsimple}, $aD=bC$ and since $C$ is primitive, $D$ is a multiple of $C$.
\end{enumerate}
\end{proof}

\begin{lemma} \label{genre} $ $
\begin{enumerate} 
\item The curve $C$ is $1$-connected \emph{(}see \cite[pp. 69]{BPVDV}\emph{)}.
\item We have $\mathrm{h}^0(C, \oo_C)=\mathrm{h}^1(C, \oo_C)=1$.
\item If $d=1$, then $C_1$ has arithmetic genus one. If $d \geq2$, all the curves $C_i$ are rational curves of self-intersection $-2$.
\end{enumerate}
\end{lemma}
\begin{proof} $ $
\begin{enumerate}
\item Let us write $C=C_1+C_2$ where $C_1$ and $C_2$ are effective and supported in $C$, with possible components in common. By Lemma \ref{base} (3), $C_1^2\leq 0$ and $C_2^2\leq 0$. Since $C^2=0$, we must have $C_1.C_2 \geq 0$. If $C_1.C_2=0$, then $C_1^2=C_2^2=0$ so that by Lemma \ref{base} (4), $C_1$ and $C_2$ are multiples of $C$, which is impossible.
\par \medskip
\item By $(1)$ and \cite[Corollary 12.3] {BPVDV}, $\mathrm{h}^0(C, \oo_C)=1$. The dualizing sheaf $\omega_C$ of $C$ is the restriction of the line bundle  $K_X+C$ to the divisor $C$. Therefore, for any integer $i$ between $1$ and $d$, $\mathrm{deg}\, ({\omega_C})_{| C_i}=(K_X+C).C_i=0$ by Lemma \ref{base} (1). Therefore, by \cite[Lemma 12.2] {BPVDV}, $\mathrm{h}^0(C, \omega_C)\leq 1$ with equality if and only if $\omega_C$ is trivial.
We can now apply the Riemann-Roch theorem for singular embedded curves \cite[Theorem 3.1]{BPVDV}: since $\omega_C$ has total degree zero, we have $\chi(\omega_C)=\chi(\oo_C)$. But using Serre duality \cite[Theorem 6.1]{BPVDV}, $\chi(\omega_C)=-\chi(\oo_C)$ so that $\chi(\oo_C)=\chi(\omega_C)=0$. It follows that $\mathrm{h}^1(C, \oo_C)=1$.
\par \medskip
\item This follows from the genus formula: $2p_a(C_i)-2=C_i^2+K_X.C_i=C_i^2<0$ so that $p_a(C_i)=0$ and $C_i^2=-2$. Now the geometric genus is always smaller than the arithmetic genus, so that the geometric genus of $C_i$ is $0$, which means that $C_i$ is rational.

\end{enumerate}
\end{proof}

We can now prove a result which will be crucial in the sequel:
\begin{proposition} \label{crucial}
Let $D$ be a divisor on $X$ such that $D. C=0$. Then there exists a positive integer $N$ and a divisor $S$ supported in $C$ such that for all $i$, $(ND-S).\, C_i=0$.
\end{proposition}
\begin{proof}
Let $V$ be the $\Q$-vector space spanned by the $C_i$'s in $\mathrm{NS}_{\Q}(X)$, by Lemma \ref{base} (3), it has dimension $r$. We have a natural morphism $\lambda \colon V \rightarrow \Q^r$ defined by $\lambda(x)=(x.C_1, \ldots, x.C_r)$. The kernel of this morphism are vectors in $V$ orthogonal to all the $C_i$'s. Such a vector is obviously isotropic, and by Lemma \ref{base} (4), it is a rational multiple of $D$. Therefore the image of $\lambda$ is a hyperplane in $\Q^r$, which is the hyperplane $\sum_i a_i x_i=0$. Indeed, for any element $x$ in $V$, we have $\sum_i a_i \, (x.C_i)=x.C=0$.
\par \medskip
Let us consider the element $w=(D.C_1, \ldots, D.C_r)$ in $\Q^r$. Since $\sum_i a_i \, (D.C_i)=D.C=0$, we have $w=\lambda(S)$ for a certain $S$ in $V$. This gives the result. 
\end{proof}
\subsection{The trace morphism} \label{5.3}
In this section, we introduce the main object in Gizatullin's proof: the \textit{trace morphism}. For this, we must use the Picard group of the embedded curve $C$. It is the moduli space of line bundles on the complex analytic space $\mathcal{O}_C$, which is $\mathrm{H}^1(C, \oo_C^{\,\times})$. 
\par \medskip
Recall \cite[Proposition 2.1]{BPVDV} that $\mathrm{H}^1(C, \Z_C)$ embeds as a discrete subgroup of $\mathrm{H}^1(C, \oo_C)$. The connected component of the line bundle $\oo_C$ is denoted by $\mathrm{Pic}^0 (C)$, it is the abelian complex Lie group $\mathrm{H}^1(C, \oo_C)/ \mathrm{H}^1(C, \Z_C)$. We have an exact sequence
\[
0 \rightarrow \mathrm {Pic}^0(C) \rightarrow \mathrm{Pic}\,(C) \xrightarrow{\mathrm{c}_1} \mathrm{H}^1(C, \Z)
\]
and $\mathrm{H}^1 (C, \mathbb{Z})\simeq \Z^d$. Therefore, connected components of $\mathrm{Pic}\,(C)$ are indexed by sequences $(n_1, \ldots, n_d)$ corresponding to the degree of the line bundle on each irreducible component of $C$.
By Lemma \ref{genre} (2), $\mathrm{Pic}^0(C)$ can be either $\C$, $\C^{\times}$, or an elliptic curve.
\par \medskip
The trace morphism is a group morphism $\mathfrak{tr} \colon \mathrm{Pic}\,(X) \rightarrow \mathrm{Pic}\, (C)$ defined by $ \mathfrak{tr}\, (\mathcal{L})=\mathcal{L}_{| C}$ . Remark that $C.C_i=0$ for any $i$, so that the line bundle $\oo_X\left(C\right)$ restricts to a line bundle of degree zero on each component $a_i \,C_i$.
\begin{proposition} $ $ \label{torsion}
\begin{enumerate}
\item The line bundle $\mathfrak{tr}\left(\oo_X (C ) \right)$ is not a  torsion point in $\mathrm{Pic}^0(C)$.
\item The intersection form is negative definite on $\mathrm{ker}\, (\mathfrak{tr})$.
\end{enumerate}
\end{proposition}
\begin{proof} $ $
\begin{enumerate}
\item Let $N$ be an integer such that $N\mathfrak{tr}\left(\oo_X (C ) \right)=0$ in $\mathrm{Pic}\,(C)$. Then we have a short exact sequence 
\[
0 \rightarrow \oo_X((N-1)C) \rightarrow \oo_X (NC) \rightarrow \oo_C \rightarrow 0.
\]
\par \smallskip
\noindent Now $\mathrm{h}^2(X, \oo_X((N-1)C))=\mathrm{h}^0(\oo_X(-(N-1)C+K_X)=0$, so that the map
\[
\mathrm{H}^1(X, \oo_X(NC)) \rightarrow \mathrm{H}^1(C, \oo_C)
\]
is onto. It follows from Lemma \ref{genre} (2) that $\mathrm{h}^1(X, \oo_X(NC)) \geq 1$ so that by Riemann-Roch
\[
\mathrm{h}^0(X, \oo_X(NC)) \geq \mathrm{h}^1(X, \oo_X(NC)) + \chi(\oo_X) \geq 2.
\]
This yields a contradiction since we have assumed that $|N\theta|=\{NC\}$.
\par \medskip
\item Let $D$ be a divisor in the kernel of $\mathfrak{tr}$. By the Hodge index theorem $D^2 \leq 0$. Besides, if $D^2=0$, then $D$ and $C$ are proportional. In that case, a multiple of $C$ would be in $\ker \,(\mathfrak{tr})$, hence $\mathfrak{tr}\, (\oo_X(C))$ would be a torsion point in $\mathrm{Pic}\,(C)$. 
\end{enumerate}
\end{proof}
\section{Proof of Gizatullin's theorem} \label{6}
\subsection{The general strategy} \label{6.1}
The strategy of the proof is simple in spirit. Let $\mathfrak{P}$ be the image of $\mathfrak{tr}$ in $\mathrm{Pic}\, (C)$, so that we have an exact sequence of abelian groups
\[
1 \rightarrow \ker \, (\mathfrak{tr}) \rightarrow \mathrm{Pic}\, (X) \rightarrow \mathfrak{P} \rightarrow 1
\]
By Proposition \ref{torsion}, the intersection form is negative definite on $\ker \,(\mathfrak{tr})$, so that $f^*$ is of finite order  on $\ker \,(\mathfrak{tr})$. In the first step of the proof, we will prove that for any divisor $D$ on $X$ orthogonal to $C$, $f^*$ induces a morphism of finite order on each connected component of any element $\mathfrak{tr} (D)$ in $\mathrm{Pic}\,(C)$. 
\par \medskip
In the second step, we will prove that the action of $f^*$ on $\mathrm{Pic} (X)$ is finite. This will give the desired contradiction.
\subsection{Action on the connected components of $\mathfrak{P}$} \label{6.2}
In this section, we prove that $f^*$ acts finitely on "many" connected components of $\mathfrak{P}.$ More precisely:\begin{proposition} \label{elliptic}
Let $D$ be in $\mathrm{Pic}\, (X)$ such that $D.C=0$, and let $\mathfrak{X}_D$ be a connected component of $\mathfrak{tr} (D)$ in $\mathrm{Pic}\, (C)$. Then the restriction of $f^*$ to $\mathfrak{X}_D$ is of finite order.
\end{proposition}
\begin{proof}
We start with the case $D=0$ so that $\mathfrak{X}=\mathrm{Pic}^0(C)$. Then three situations can happen:
\par \smallskip
-- If $\mathrm{Pic}^0 (C)$ is an elliptic curve, then its automorphism group is finite (by automorphisms, we mean group automorphisms).
\par \smallskip
-- If $\mathrm{Pic}^0 (C)$ is isomorphic to $\C^{\times}$, its automorphism group is $\{\mathrm{id}, z \rightarrow z^{-1}\}$, hence of order two, so that we can also rule out this case.
\par \smallskip
-- Lastly, if $\mathrm{Pic}^0 (C)$ is isomorphic to $\C$, its automorphism group is $\C^{\times}$. We know that $C$ is a non-zero element of $\mathrm{Pic}^0(C)$ preserved by the action of $f^*$. This forces $f^*$ to act trivially on $\mathrm{Pic}^0 (C)$.
\par \medskip
Let $D$ be a divisor on $X$ such that $D.C=0$. By Proposition \ref{crucial}, there exists a positive integer $N$ and a divisor $S$ supported in $C$ such that $N \mathfrak{tr}\,(D)- \mathfrak{tr}\,(S) \in \mathrm{Pic}^0(C)$. Let $m$ be an integer such that $f^m$ fixes the components of $C$ and acts trivially on $\mathrm{Pic}\,(C)$. We define a map $\lambda \colon \Z \rightarrow \mathrm{Pic}^0(C)$ by the formula
\[
\lambda(k)=(f^{km})^* \{\mathfrak{tr} (D)\}-\mathfrak{tr} (D)
\]
\par \smallskip

\begin{enumerate}
\item[] \textbf{Claim 1}: $\lambda$ does not depend on $D$.
\par
Indeed, if $D'$ is in $\mathfrak{X}_D$, then $\mathfrak{tr}\,(D'-D) \in \mathrm{Pic}^0(C)$ so that
\[
(f^{km})^*(D'-D)=D'-D.
\]
This gives $(f^{km})^* \{\mathfrak{tr} (D')\}-\mathfrak{tr} (D')=(f^{km})^* \{\mathfrak{tr} (D)\}-\mathfrak{tr} (D)$
\par \medskip
\item[] \textbf{Claim 2}: $\lambda$ is a group morphism.
\begin{alignat*}{3}
\lambda(k+l)
& = (f^{km})^*(f^{lm})^* \{\mathfrak{tr} (D)\}-\mathfrak{tr} (D) \\
& = \begin{aligned}[t]
      (f^{km})^*\left\{(f^{lm})^* \{\mathfrak{tr} (D)\} \right\}
      & -\left\{(f^{lm})^* \{\mathfrak{tr} (D)\} \right\} \\
      & + (f^{lm})^* \{\mathfrak{tr} (D)\}-\mathfrak{tr} (D)
   \end{aligned} \\
&= \lambda(k) + \lambda(l) \quad \textrm{by Claim 1}.
\end{alignat*}
\par \medskip
\item[] \textbf{Claim 3}: $\lambda$ has finite image.
\par
For any integer $k$, since $N \,\mathfrak{tr}\,(D)- \mathfrak{tr}\,(S) \in \mathrm{Pic}^0(C)$, $(f^{km})^* \{N\, \mathfrak{tr} (D) \}=  N \,\mathfrak{tr} (D)$. Therefore, we see that $(f^{km})^* \{ \mathfrak{tr} (D) \} - \mathfrak{tr} (D)=\lambda(k)$ is a $N$-torsion point in $\mathrm{Pic}^0(C)$. Since there are finitely many $N$-torsion points, we get the claim.
\end{enumerate}
\par \medskip
We can now conclude. By claims $2$ and $3$, there exists an integer $s$ such that the restriction of $\lambda$ to $s \Z$ is trivial. This implies that $D$ is fixed by $f^{ms}$. By claim 1, all elements in $\mathfrak{X}_D$ are also fixed by $f^{ms}$.
\end{proof}
\subsection{Lift of the action from $\mathfrak{P}$ to the Picard group of $X$} \label{6.3}
By Proposition \ref{torsion} (2) and Proposition \ref{elliptic}, up to replacing $f$ with an iterate, we can assume that $f$ acts trivially on all components $\mathfrak{X}_D$, on $\mathrm{ker}\, (\mathfrak{tr})$, and fixes the components of $C$.
\par \medskip
Let $r$ be the rank of  $\mathrm{Pic}\, (X)$, and fix a basis $E_1, \ldots, E_r$ of $\mathrm{Pic}\, (X)$ composed of irreducible reduced curves. Let $n_i=E_i.C$. If $n_i=0$, then either $E_i$ is a component of $C$, or $E_i$ is disjoint from $C$. In the first case $E_i$ is fixed by $f$. In the second case, $E_i$ lies in the kernel of $\mathfrak{tr}$, so that it is also fixed by $f$. 
\par \medskip
Up to re-ordering the $E_i$'s, we can assume that $n_i>0$ for $1 \leq i \leq s$ and $n_i=0$ for $i>s$. We put $m=n_1 \ldots n_s$, $m_i=\frac{m}{n_i}$ and $L_i=m_iE_i$. 

\begin{proposition} \label{chic}
For $1 \leq i \leq s$, $L_i$ is fixed by an iterate of $f$.
\end{proposition}
\begin{proof}
For $1 \leq i \leq s$, we have $L_i.C=m$, so that for $1\leq i, j \leq s $, $(L_i-L_j).C=0$. Therefore, by Proposition \ref{elliptic}, an iterate of $f$ acts trivially on $\mathfrak{X}_{L_i-L_j}$. Since there are finitely many couples $(i,j)$, we can assume (after replacing $f$ by an iterate) that $f$ acts trivially on all $\mathfrak{X}_{L_i-L_j}$. 
\par \medskip
Let us now prove that $f^* L_i$ and $L_i$ are equal in $\mathrm{Pic}\,(X)$. Since $f^*$ acts trivially on the component $\mathfrak{X}_{L_i-L_j}$, we have $\mathfrak{tr}\,(f^*L_i-L_i)=\mathfrak{tr}\,(f^*L_j-L_j)$. Let $D=f^*L_1-L_1$. Then for any $i$, we can write $f^*L_i-L_i=D+D_i$ where $\mathfrak{tr}\, (D_i)$=0.
\par \medskip
Let us prove that the class $D_i$ in $\mathrm{Pic}\, (X)$ is independent of $i$. For any element $A$ in $\mathrm{ker}\, (\mathfrak{tr})$, we have
\[
D_i. \,A=(f^*L_i-L_i-D). \,A=f^*L_i . f^*A-L_i . \,A-D. \,A=-D. \,A
\]
since $f^*A=A$. Now since the intersection form in non-degenerate on $\mathrm{ker}\,(\mathfrak{tr})$, if $(A_k)_k$ is an orthonormal basis of $\mathrm{ker}\,(\mathfrak{tr})$, 
\[
D_i=-\sum_k (D_i . \, A_k) \,A_k=\sum_k (D . \, A_k) \, A_k.
\] 
Therefore, all divisors $D_i$ are linearly equivalent. Since $D_1=0$, we are done.
\end{proof}
We can end the proof of Gizatullin's theorem. Since $L_1, \ldots, L_s, E_{s+1}, \ldots, E_{r}$ span $\mathrm{Pic}\, (X)$ over $\Q$, we see that the action of $f$ on $\mathrm{Pic}\,(X)$ is finite. This gives the required contradiction.

\section{Minimal rational elliptic surfaces} \label{7}
Throughout this section, we will assume that $X$ is a rational elliptic surface whose fibers contain no exceptional curves; such a surface will be called by a slight abuse of terminology a minimal elliptic rational surface.
\subsection{Classification theory} \label{7.1}
The material recalled in this section is more or less standard, we refer to \cite[Chap. II \S 10.4]{DS} for more details.
\begin{lemma} \label{hehehe}
Let $X$ be a rational surface with $K_X^2=0$. Then $|-K_X| \neq \varnothing$. Besides, for any divisor $\mathfrak{D}$ in $|-K_X|$ \emph{:}
\begin{enumerate}
\item $\mathrm{h}^1(\mathfrak{D}, \oo_{\mathfrak{D}})=1$.
\item For any divisor $D$  such that $0< D < \mathfrak{D}$, $\mathrm{h}^1({D}, \oo_{D})=0$.
\item $\mathfrak{D}$ is connected and its class is non-divisible in $\mathrm{NS}\,(X)$.
\end{enumerate}
\end{lemma}
\begin{proof} $ $
The fact that $|-K_X| \neq \varnothing$ follows directly from the Riemann-Roch theorem.
\begin{enumerate}
\item We write the exact sequence of sheaves
\[
0 \rightarrow \oo_X(-\mathfrak{D}) \rightarrow \oo_X \rightarrow \oo_{\mathfrak{D}} \rightarrow 0.
\]
Since $X$ is rational,  $\mathrm{h}^1(X, \oo_X)=\mathrm{h}^2(X, \oo_X)=0$; and since $\mathfrak{D}$ is an anticanonical divisor, we have by Serre duality
\[
\mathrm{h^2(X, -\mathfrak{D})}=\mathrm{h^0(X, K_X)}=1.
\]
\item We use the same proof as in (1) with $D$ instead of $\mathfrak{D}$. We have 
\[
\mathrm{h^2(X, -{D})}=\mathrm{h^0(X, K_X+D)}=\mathrm{h^0(X, D-\mathfrak{D})}=0.
\]
\item The connectedness follows directly from $(1)$ and $(2)$: if $\mathfrak{D}$ is the disjoint reunion of two divisors $\mathfrak{D}_1$ and $\mathfrak{D}_2$, then $\mathrm{h}^0(\mathfrak{D}, \oo_\mathfrak{D})=\mathrm{h}^0(\mathfrak{D}_1, \oo_{\mathfrak{D}_1})+\mathrm{h}^0(\mathfrak{D}_2, \oo_{\mathfrak{D}_2})=0$, a contradiction. 
\par \medskip
\noindent Assume now that $\mathfrak{D}=m \mathfrak{D}'$ in $\mathrm{NS}\,(X)$, where $\mathfrak{D}'$ is not necessarily effective and $m \geq 2$. Then, using Riemann-Roch, 
\[
\qquad\mathrm{h}^0(X, {\mathfrak{D}'})+\mathrm{h}^0(X, -(m+1) \mathfrak{D}')\geq 1.
\]
If $-(m+1) \mathfrak{D}'$ is effective, then $|NK_X|\neq \varnothing$ for some positive integer $N$, which is impossible. Therefore the divisor $\mathfrak{D}'$ is effective; and $\mathfrak{D}-\mathfrak{D'}=(m-1) \mathfrak{D'}$ is also effective.
Using Riemann-Roch one more time,
\[
\begin{aligned}
\qquad\mathrm{h}^0(\mathfrak{D}', \oo_{\mathfrak{D}'})-\mathrm{h}^1(\mathfrak{D}', \oo_{\mathfrak{D}'})=\chi(\oo_{\mathfrak{D}'})&=\chi(\oo_X)-\chi(\mathcal{O}_X(-\mathfrak{D}'))\\
&=-\frac{1}{2} \mathfrak{D}' .(\mathfrak{D}'+K_X)=0.
\end{aligned}
\]
Thanks to $(2)$, since $0 < \mathfrak{D'} < \mathfrak{D}$, 
$\mathrm{h}^1(\mathfrak{D}', \oo_{\mathfrak{D}'})=0$, so that $\mathrm{h}^0(\mathfrak{D}', \oo_{\mathfrak{D}'})=0$. This gives again a contradiction.

\end{enumerate}

\end{proof}

\begin{proposition} \label{dix} Let $X$ be a rational minimal elliptic surface and $C$ be a smooth fiber.
\begin{enumerate}
\item $K_X^2=0$ and $\mathrm{rk} \, \{\mathrm{Pic}\,(X)\}=10.$
\item For any irreducible component $E$ of a reducible fiber , $E^2 <0$ and $E.K_X=0.$
\item There exists a positive integer $m$ such that $-mK_X=C$ in $\mathrm{Pic}\,(X)$.
\end{enumerate}
\end{proposition}
\begin{proof}
Let $C$ be any fiber of the elliptic fibration. Then for any reducible fiber $D=\sum_{i=1}^s a_i D_i$, $D_i . C=C^2=0$. By the Hodge index theorem, $D_i^2 \leq 0$. If $D_i^2=0$, then $D_i$ is proportional to $C$. Let us write $D=a_i D_i+(D-a_i D_i)$. On the one hand, $a_i D_i .(D-a_i D_i)=0$ since $D_i$ and $D-D_i$ are proportional to $C$. On the other hand, $a_i D_i .(D-a_i D_i)>0$ since $D$ is connected. This proves the first part of (2).
\par \medskip
We have $K_X.C=C.C=0$. By the Hodge index theorem, $K_X^2 \leq 0$. We have an exact sequence
\[
0 \rightarrow K_X \rightarrow K_X+C \rightarrow \omega_C \rightarrow 0.
\]
\par \medskip
Since $\mathrm{h}^0(C, \omega_C)=1$ and $\mathrm{h^0(X, K_X)}=\mathrm{h}^1(X, K_X)=\mathrm{h}^1(X, \oo_X)=0$, $\mathrm{h}^0(X, K_X+C)=1$. Thus, the divisor $D=K_X+C$ is effective. Since $D.C=0$, all components of $D$ are irreducible components of the fibers of the fibration. The smooth components cannot appear, otherwise $K_X$ would be effective. Therefore, if $D=\sum_{i=1}^{s} a_i D_i$, we have $D_i^2<0$. Since $X$ is minimal, $K_X.D_i \geq 0$ (otherwise $D_i$ would be exceptional). Thus, $K_X.D \geq 0$. 
\par \medskip
Since $C$ is nef, we have $D^2=(K_X+C).D \geq K_X.D\geq0$. On the other hand, $D.C=0$ so that $D^2=0$ by the Hodge index theorem. Thus $K_X^2=0$. Since $X$ is rational, it follows that $\mathrm{Pic}\, (X)$ has rank $10$. This gives (1).
\par \medskip
Now $K_X^2=C^2=C.K_X=0$ so that $C$ and $K_X$ are proportional. By Lemma \ref{hehehe}, $K_X$ is not divisible in $\mathrm{NS}\,(X)$, so that $C$ is a multiple of $K_X$. Since $|dK_X|=0$ for all positive $d$, $C$ is a negative multiple of $K_X$. This gives (3).
\par \medskip
The last point of (2) is now easy: $E.K_X=-\frac{1}{m} E.C=0$. 
\end{proof}
We can be more precise and describe explicitly the elliptic fibration in terms of the canonical bundle. 
\begin{proposition} \label{primitive}
Let $X$ be a minimal rational elliptic surface. Then for $m$ large enough, we have $\mathrm{dim}\, |-mK_X| \geq 1$. For $m$ minimal with this property, $|-mK_X|$ is a pencil without base point whose generic fiber is a smooth and reduced elliptic curve. 
\end{proposition}
\begin{proof}
The first point follows from Proposition \ref{dix}.
Let us prove that $|-mK_X|$ has no fixed part. As usual we write $-mK_X=F+D$ where $F$ is the fixed part. Then since $C$ is nef and proportional to $K_X$, $C.F=C.D=0$. Since $D^2\geq 0$, by the Hodge index theorem $D^2=0$ and $D$ is proportional to $C$. Thus $D$ and $F$ are proportional to $K_X$. 
\par \medskip
By Lemma \ref{hehehe}, the class of $K_X$ is non-divisible in $\mathrm{NS}\,(X)$. Thus $F=m' \mathfrak{D}$ for some integer $m'$ with $0 \leq m' <m$. Hence $D=(m-m') \,\mathfrak{D}=-(m-m') \,K_X$ and $\mathrm{dim}\, |D| \geq 1$. By the minimality of $m$, we get $m'=0$.
\par \medskip
Since $K_X^2=0$, $-mK_X$ is basepoint free and $|-mK_X|=1$. Let us now prove that the divisors in $|-mK_X|$ are connected. If this is not the case, we use the Stein decomposition and write the Kodaira map of $-mK_X$ as
\[
X \rightarrow S \xrightarrow{\psi} |-mK_X|^*
\]
where $S$ is a smooth compact curve, and $\psi$ is finite. Since $X$ is rational, $S=\mathbb{P}^1$ and therefore we see that each connected component $D$ of a divisor in $|-mK_X|$ satisfies $\mathrm{dim}\, |D| \geq 1$. Thus $\mathrm{dim}\, |D| \geq 2$ and we get a contradiction.
\par \medskip
We can now conclude: a generic divisor in $|-mK_X|$ is smooth and reduced by Bertini's theorem. The genus formula shows that it is an elliptic curve.
\end{proof}
\begin{remark}$ $ \par
\begin{enumerate}
\item Proposition \ref{primitive} means that the relative minimal model of $X$ is a \textit{Halphen surface} of index $m$, that is a rational surface such that $|-mK_X|$ is a pencil without fixed part and base locus. Such a surface is automatically minimal.
\item The elliptic fibration $X \rightarrow |-mK_X|^*$ doesn't have a rational section if $m \geq 2$. Indeed, the existence of multiple fibers (in our situation, the fiber $m \mathfrak{D}$) is an obstruction for the existence of such a section.
\end{enumerate}
\end{remark}
\subsection{Reducible fibers of the elliptic fibration} \label{6.2}
We keep the notation of the preceding section: $X$ is a Halphen surface of index $m$ and $\mathfrak{D}$ is an anticanonical divisor. 
\begin{lemma} \label{woodloot}
All the elements of the system $|-mK_X|$ are primitive, except the element $m \mathfrak{D}$.
\end{lemma}
\begin{proof}
Since $K_X$ is non-divisible in $\mathrm{NS}\,(X)$, a non-primitive element in $|-mK_X|$ is an element of the form $k D$ where $D \in |m' \mathfrak{D}|$ and $m=k m'$. But $\mathrm{dim} \,|m' \mathfrak{D}|=0$ so that $|D|=|m' \mathfrak{D}|=\{m' \mathfrak{D}\}$. 
\end{proof}
In the sequel, we denote by $S_1, \ldots, S_\lambda$ the reducible fibers of $|-mK_X|$.
We prove an analog of Lemma \ref{base}, but the proofs will be slightly different.
\begin{lemma} \label{libre} $ $
\begin{enumerate}
\item Let $S=\alpha_1 E_1 + \ldots + \alpha_{\nu} E_{\nu}$ be a reducible fiber of $|-mK_X|$. Then the classes of the components $E_i$ in $\mathrm{NS}\,(X)$ are linearly independent.
\item If $D$ is a divisor supported in $S_1 \cup \ldots \cup S_{\lambda}$ such that $D^2=0$, then there exists integers $n_i$ such that $D=n_1S_1 + \ldots + n_{\lambda} S_{\lambda}$.
\end{enumerate}
\end{lemma}
\begin{proof} If there is a linear relation among the curves $E_i$, we can write it as $D_1 \equiv D_2$, where $D_1$ and $D_2$ are linear combinations of the $E_i$ with positive coefficients (hence effective divisors) having no component in common. We have $D_1^2=D_1. \, D_2 \geq 0$. On the other hand $S.\, D_1=0$ and $D^2=0$, so by the Hodge index theorem $S$ and $D_1$ are proportional. Let $E$ be a component of $S$ intersecting $D_0$ but not included in $D_0$. If $a \,D_1 \sim b \,S$, then $0=b \,S.\,E=a \,D_1.\, E>0$, and we are done.
\par \medskip
For the second point, let us write $D=D_1+ \ldots +D_{\lambda}$ where each $D_i$ is supported in $S_i$. Then the $D_i$'s are mutually orthogonal. Besides, $D_i. C=0$, so that by the Hodge index theorem $D_i^2 \leq 0$. Since $D^2=0$, it follows that $D_i^2=0$ for all $i$. 
\par \medskip
We pick an $i$ and write $D_i=D$ and $S_i=S$. Then there exists integers $a$ and $b$ such that $a D \sim b S$. Therefore, if $D=\sum \beta_q \,E_q$,  $\sum_q (a \alpha_q-b \beta_q) \,E_q=0 $ in $\mathrm{NS}\,(X)$. By Lemma \ref{libre}, $a \alpha_q-b \beta_q=0$ for all $q$, so that $b$ divides $a \alpha_q$ for all $q$. By Lemma \ref{woodloot}, $b$ divides $a$. If $b=ac$, then $\beta_q=c \alpha_q$ for all $q$, so that $D=cS$. 
\end{proof}
\par \medskip
Let $\rho \colon X \rightarrow \mathbb{P}^1$ be the Kodaira map of $|-mK_X|$, and $\xi$ be the generic point of $\mathbb{P}^1$. We denote by $\mathfrak{X}$ the algebraic variety $\rho^{-1}(\xi)$, which is a smooth elliptic curve over the field $\mathbb{C}(t)$.
\par \medskip
Let $\mathcal{N}$ be the kernel of the natural restriction map ${\mathfrak{t}} \colon \mathrm{Pic}\,(X) \rightarrow\mathrm{Pic}\,(\mathfrak{X})$. The image of $\mathfrak{t}$ is the set of divisors on $\mathfrak{X}$ defined over the field $\mathbb{C}(t)$, denoted by $\mathrm{Pic}\, (\mathfrak{X}/\C(t))$. 
\par \medskip
The algebraic group $\mathrm{Pic}_0(\mathfrak{X})$ acts naturally on $\mathfrak{X}$, and this action is simple and transitive over any algebraic closure of $\C(t)$.

\begin{proposition} \label{sept}
If $S_1, \ldots, S_{\lambda}$ are the reducible fibers of the pencil $|-mK_X|$ and $\mu_j$ denotes the number of components of each curve $S_j$, then
$
\mathrm{rk}\, \mathcal{N}=1+\sum_{i=1}^{\lambda} \,\{\mu_i- 1\}.
$
\end{proposition}

\begin{proof}
The group $\mathcal{N}$ is generated by $\mathfrak{D}$ and the classes of the reducible components of $|-mK_X|$.
\par \medskip
We claim that the module of relations between these generators is generated by the relations 
$\alpha_1 [E_1] + \ldots + \alpha_{\nu} [E_{\nu}]=m[\mathfrak{D}]$ where $\alpha_1 E_1 + \cdots + \alpha_s E_s$ is a reducible member of $|-mK_X|$. 
\par \medskip
Let $D$ be of the form $a \mathfrak{D}+D_1+ \cdots + D_{\lambda}$ where each $D_i$ is supported in $S_i$, and assume that $D \sim 0$. Then $(D_1+ \cdots + D_{\lambda})^2=0$. Thanks to Lemma \ref{libre} (2), each $D_i$ is equal to $n_i S_i$ for some $n_i$ in $\Z$. Then $a+m \,\{\sum_{i=1}^{\lambda} n_i \}=0$, and
\[
a \mathfrak{D}+D_1+ \cdots + D_{\lambda}=\sum_{i=1}^{\lambda} n_i \,(S_i-m \mathfrak{D}).
\]
We also see easily that these relations are linearly independent over $\Z$. Thus, since the number of generators is $1+\sum_{i=1}^{\lambda} \mu_i$, we get the result.
\end{proof}
\begin{corollary} \label{sympa}
We have the inequality $\sum_{i=1}^{\lambda} \,\{\mu_i-1\} \leq 8$. Besides, if $\sum_{i=1}^{\lambda} \,\{\mu_i-1\}=8$, every automorphism of $X$ acts finitely on $\mathrm{NS}\,(X)$.
\end{corollary}
\begin{proof} We remark that $\mathcal{N}$ lies in ${K}_X^{\perp}$, which is a lattice of rank $9$ in $\mathrm{Pic}\,(X)$. This yields the inequality $\sum_{i=1}^{\lambda} \,(\mu_i-1) \leq 8$. 
\par \medskip
Assume $\mathcal{N}=K_X^{\perp}$, and let $f$ be an automorphism of $X$. Up to replacing $f$ by an iterate, we can assume that $\mathcal{N}$ is fixed by $f$. Thus $f^*$ is a parabolic translation leaving the orthogonal of the isotropic invariant ray $\mathbb{R} K_X$ pointwise fixed. It follows that $f$ acts trivially on $\mathrm{Pic}\,(X)$.
\end{proof}
Lastly, we prove that there is a major dichotomy among Halphen surfaces. Since there is no proof of this result in Gizatullin's paper, we provide one for the reader's convenience. 
\par \medskip
Let us introduce some notation: let $\mathrm{Aut}_0(X)$ be the connected component of $\mathrm{id}$ in $\mathrm{Aut}\, (X)$ and $\widetilde{\mathrm{Aut}}\, (X)$ be the group of automorphisms of $X$ preserving fiberwise the elliptic fibration.
\begin{proposition}[see {\cite[Prop. B]{GIZ}}] \label{waza}
Let $X$ be a Halphen surface. Then $X$ has at least two degenerate fibers. The following are equivalent:
\begin{enumerate}
\item[(i)] $X$ has \textit{exactly} two degenerate fibers.
\item[(ii)] $\mathrm{Aut}_0(X)$ is an algebraic group of positive dimension.
\item[(iii)] $\widetilde{\mathrm{Aut}}\, (X)$ has infinite index in $\mathrm{Aut}\, (X)$.
\end{enumerate}
Under any of these conditions, $\mathrm{Aut}_0(X) \simeq\C^{\times}$, and $\widetilde{\mathrm{Aut}}\, (X)$ is finite, and $\mathrm{Aut}_0(X)$ has finite index in $\mathrm{Aut}\, (X)$.
\end{proposition}

\begin{proof}
Let $\mathcal{Z}$ be the finite subset of $\mathbb{P}^1$ consisting of points $z$ such that $\pi$ is not smooth at some point of the fiber $X_z$, and $U$ be the complementary set of $\mathcal{Z}$ in $\mathbb{P}^1$. The points of $\mathcal{Z}$ correspond to the degenerate fibers of $X$.
\par \medskip
Let $\mathcal{M}_1$ be the moduli space of elliptic curves, considered as a complex orbifold. It is the quotient orbifold $\mathfrak{h} / \mathrm{SL}(2; \mathbb{Z})$ and its coarse moduli space $|\mathcal{M}_1|$ is $\C$. The elliptic surface over $U$ yields a morphism of orbifolds $\phi \colon U \rightarrow \mathcal{M}_{1}$, hence a morphism $| \phi | \colon U \rightarrow \mathbb{C}$. The orbifold universal cover of $\mathcal{M}_{1}$ is $\mathfrak{h}$, so that $|\phi|$ induces a holomorphic map $\widetilde{U} \rightarrow \mathfrak{h}$.
\par \medskip
If $\# \mathcal{Z} \in \{0, 1, 2\}$, then $\widetilde{U}=\mathbb{P}^1$ or $\widetilde{U}=\C$ and $|\phi|$ is constant. This means that all fibers of $X$ over $U$ are isomorphic to a fixed elliptic curve $E$. 
\par \medskip
Let $H$ be the isotropy group of $\mathcal{M}_1$ at $E$, it is a finite group of order $2$, $4$ or $6$. Then $\phi$ factorizes as the composition $U \rightarrow {B}H \rightarrow \mathcal{M}_1$ where ${B} H$ is the orbifold $\bullet_{H}$. The stack morphisms from $U$ to ${B}H$ are simply $H$-torsors on $U$, and are in bijection with $\mathrm{H}^1(U, H)$.
\par \medskip
In the case $\# \mathcal{Z} \in \{0, 1\}$, that is $U=\mathbb{P}^1$ or $U=\mathbb{C}$, then $\mathrm{h}^1(U, H)=0$. Thus $X$ is birational to $E \times \mathbb{P}^1$ which is not possible for rational surfaces. This proves the first part of the theorem.
\par \medskip
(iii) $\Rightarrow$ (i)
 We have an exact sequence
\[
0 \rightarrow \widetilde{\mathrm{Aut}}\,(X) \rightarrow \mathrm{Aut}\,(X) \xrightarrow{\kappa} \mathrm{Aut}\, (\mathbb{P}^1)
\]
The image of $\kappa$ must leave the set $\mathcal{Z}$ globally fixed. If $\# \mathcal{Z} \geq 3$, then the image of $\kappa$ is finite, so that $\widetilde{\mathrm{Aut}}\,(X)$ has finite index in $\mathrm{Aut}\,(X)$.
\par \medskip
(i) $\Rightarrow$ (ii) In this situation, we deal with the case $U=\C^{\times}$. The group $\mathrm{H}^1(\mathbb{C}^{\times}, H)$ is isomorphic to $H$. For any element $h$ in $H$, let $n$ be the order of $h$ and $\zeta$ be a $n$-th root of unity. The cyclic group $\Z / n \Z$ acts on $\mathbb{C}^{\times} \times E$ by the formula $p.(z, e)=(\zeta^{p} z, h^p . e)$. The open elliptic surface over $\mathbb{C}^{\times}$ associated with the pair $(E, h)$ is the quotient of $\mathbb{C}^{\times} \times E$ by $\Z/n\Z$.
We can compactify everything: the elliptic surface associated with the pair $(E, h)$ is obtained by desingularizing the quotient of $\mathbb{P}^1 \times E$ by the natural extension of the $\Z / n\Z$-action defined formerly. By this construction, we see that the $\mathbb{C}^{\times}$ action on $\pi^{-1}(U)$ extends to $X$. Thus $\mathrm{Aut}_0(X)$ contains $\C^{\times}$.
\par \medskip
(i) $\Rightarrow$ (iii) We have just proven in the previous implication that if $X$ has two degenerate fibers, then the image of $\kappa$ contains $\C^{\times}$. Therefore $\widetilde{\mathrm{Aut}}\,(X)$ has infinite index in $\mathrm{Aut}\,(X)$.
\par \medskip
(ii) $\Rightarrow$ (i) We claim that $\widetilde{\mathrm{Aut}}\, (X)$ is countable. Indeed, $\widetilde{\mathrm{Aut}}\, (X)$ is a subgroup of $\mathrm{Aut}\, (\mathfrak{X}/ \C(t))$ which contains $\mathrm{Pic}\, (\mathfrak{X}/ \C(t))$ as a finite index subgroup; and $\mathrm{Pic}\, (\mathfrak{X}/ \C(t))$ is a quotient of $\mathrm{Pic}\,(X)$ which is countable since $X$ is rational. Therefore, if $\mathrm{Aut}_0(X)$ has positive dimension, then the image of $\kappa$ is infinite. The morphism $|\phi| \colon U \rightarrow \mathbb{C}$ is invariant by the action of $\mathrm{im}\, (\kappa)$, so it must be constant. As we have already seen, this implies that $X$ has two degenerate fibers.
\par \medskip
It remains to prove the last statement of the Proposition. Since $\widetilde{\mathrm{Aut}}\,(X)$ is a countable group, $\widetilde{\mathrm{Aut}}\,(X) \cap \mathrm{Aut}_0 (X)=\{ \mathrm{id} \}$. Thus, $\mathrm{Aut}_0(X) \simeq \kappa \left(\mathrm{Aut}_0(X) \right) \simeq \mathbb{C}^{\times}$. Let $\varepsilon$ denote the natural representation of $\mathrm{Aut}\,(x)$ in $\mathrm{NS}(X)$. Since $\mathrm{Aut}_0(X) \subset \mathrm{ker}\,(\varepsilon)$, $\mathrm{ker}\,(\varepsilon)$ is infinite. Thanks to \cite{HH}, $\mathrm{im}(\varepsilon)$ is finite. To conclude, it suffices to prove that $\mathrm{Aut}_0(X)$ has finite index in $\mathrm{ker}\,(\varepsilon)$. Any smooth curve of negative self-intersection must be fixed by $\mathrm{ker}\,(\varepsilon)$. Let $\mathbb{P}^2$ be the minimal model of $X$ (which is either $\mathbb{P}^2$ or $\mathbb{F}_n$) and write $X$ as the blowup of $\mathbb{P}^2$ along a finite set $Z$ of (possibly infinitly near) points. Since $\mathrm{Aut}_0(\mathbb{P}^2)$ is connected, $\mathrm{ker}\,(\varepsilon)$ is the subgroup of elements of $\mathrm{Aut}\,(\mathbb{P}^2)$ fixing $Z$. This is a closed algebraic subgroup of $\mathrm{Aut}\,(\mathbb{P}^2)$, so $\mathrm{ker}\,(\varepsilon)_0$ has finite index in $\mathrm{ker}\,(\varepsilon)$. Since $\mathrm{ker}\,(\varepsilon)_0=\mathrm{Aut}_0(X)$, we get the result.

\end{proof}
\begin{remark}
Minimal elliptic surfaces with two degenerate fibers are called Gizatullin surfaces, they are exactly the rational surfaces possessing a nonzero regular vector field. They are Halphen surfaces of index $1$, their detailed construction is given in \cite[\S 2]{GIZ}. They have two reducible fibers $S_1$ and $S_2$ which satisfy $\mu_1+ \mu_2=10$, and $\mathrm{Aut}_0(X)$ has always finite index in $\mathrm{Aut}\,(X)$.
\end{remark}

\subsection{The main construction} \label{7.3}
In this section, we construct explicit parabolic automorphisms of Halphen surfaces.
\begin{theorem} \label{penible}
Let $X$ be a Halphen surface such that $\sum_{i=1}^{\lambda} \,\{\mu_i-1\} \leq 7$. Then there exists a free abelian group $G$ of finite index in ${\mathrm{Aut}}\,(X)$ of rank $8-\sum_{i=1}^{\lambda} \,\{\mu_i-1\}$ such that any non-zero element in $G$ is a parabolic automorphism acting by translation on each fiber of the fibration.
\end{theorem}
\begin{proof}Let $\widetilde{\mathrm{Aut}}(X)$ be the subgroup of $\mathrm{Aut}\,(X)$ corresponding to automorphisms of $X$ preserving the elliptic fibration fiberwise. By \cite[Chap. II \S 10 Thm.1]{DS}, any automorphism of $\mathfrak{X}$ defined over $\mathbb{C}(t)$ extends to an automorphism of $X$. Thus $\widetilde{\mathrm{Aut}}\,(X)=\mathrm{Aut} (\mathfrak{X}/ \mathbb{C}(t))$.
\par \medskip
Since $\mathfrak{X}$ is a smooth elliptic curve, $\mathrm{Pic}_0 \{\mathfrak{X} / \mathbb{C}(t)\}$ has finite index in $\mathrm{Aut} (\mathfrak{X}/ \mathbb{C}(t))$, so that $\mathrm{Pic}_0 \{\mathfrak{X} / \mathbb{C}(t)\}$ has finite index in 
$\widetilde{\mathrm{Aut}}\,(X)$. 
\par \medskip

The trace morphism $\mathfrak{t} \colon \mathrm{Pic}\,(X) \rightarrow \mathrm{Pic} \{\mathfrak{X} / \mathbb{C}(t)\}$ is surjective and for any divisor $D$ in $\mathrm{Pic}\,(X)$ we have $\mathrm{deg}\, \mathfrak{t}(D)=D.C$.
Therefore 
\[
K_X^{\perp}/ \mathcal{N} \simeq \mathrm{Pic}_0 \{\mathfrak{X} / \mathbb{C}(t)\} \hookrightarrow \widetilde{\mathrm{Aut}}\,(X)
\]
where the image of the last morphism has finite index. By Proposition \ref{sept}, the rank of $\mathcal{N}$ is $\sum_{i=1}^{\lambda} (\mu_i-1) +1$, which is smaller that $8$. Let $G$ be the torsion-free part of ${K_X^{\perp}}/{\mathcal{N}}$; the rank of $G$ is at least one. Any $g$ in $G$ acts by translation on the generic fiber $\mathfrak{X}$ and this translation is of infinite order in $\mathrm{Aut}\, (\mathfrak{X})$. Beside, via the morphism $\mathrm{Pic}\,(X) \rightarrow \mathrm{Pic}\,(\mathfrak{X})$, $g$ acts by translation by $\mathfrak{tr}\,(g)$ on $\mathrm{Pic}\,(\mathfrak{X})$, so that the action of $g$ on $\mathrm{Pic}\,(X)$ has infinite order.

\par \medskip
Let $g$ in $G$, and let $\lambda$ be an eigenvalue of the action of $g$ on $\mathrm{Pic}\, (X)$, and assume that $|\lambda| > 1$. If $g^*v=\lambda v$, then $v$ is orthogonal to $K_X$ and $v^2=0$. It follows that $v$ is collinear to $K_X$ and we get a contradiction. Therefore, $g$ is parabolic.
\par \medskip
To conclude the proof it suffices to prove that $ \widetilde{\mathrm{Aut}}\,(X)$ has finite index in ${\mathrm{Aut}}\,(X)$. Assume the contrary. Then Proposition \ref{waza} implies that $X$ has two degenerate fibers, that is $X$ is a Gizatullin surface. In that case $\mu_1+\mu_2=10$ (by the explicit description of Gizatullin surfaces) and we get a contradiction.
\end{proof}
\begin{corollary} \label{hapff}
Let $X$ be a Halphen surface. The following are equivalent:
\begin{enumerate}
\item[(i)] $\sum_{i=1}^{\lambda} \{\mu_i-1\}=8$.
\item[(ii)] The group $\widetilde{\mathrm{Aut}}(X)$ is finite.
\item[(iii)] The image of $\mathrm{Aut}\, (X)$ in $\mathrm{GL}\!\left(\mathrm{NS}\,(X)\right)$ is finite.
\end{enumerate}
\end{corollary}
\begin{proof}
(i) $\Leftrightarrow$ (ii) Recall (see the proof of Proposition \ref{penible}) that $K_X^{\perp}/\mathcal{N}$ has finite index in $\widetilde{\mathrm{Aut}}\, (X)$. This gives the equivalence between (i) and (ii) since $K_X^{\perp}/\mathcal{N}$ is a free group of rank $8-\sum_{i=1}^{\lambda} \{\mu_i-1\}$.
\par \medskip
(i) $\Rightarrow$ (iii) This is exactly Corollary \ref{sympa}.
\par \medskip
(iii) $\Rightarrow$ (i) Assume that $\sum_{i=1}^{\lambda} \{\mu_i-1\} \leq 7$. Then $X$ carries parabolic automorphisms thanks to Theorem \ref{penible}. This gives the required implication.
\end{proof}
Let us end this section with a particular but illuminating example: \textit{unnodal Halphen surfaces}. By definition, an unnodal Halphen surface is a Halphen surface without reducible fibers. In this case, $\mathcal{N}$ is simply the rank one module $\Z K_X$, so that we have an exact sequence
\[
0 \rightarrow \Z K_X \rightarrow K_X^{\perp} \underset{\lambda}{\hookrightarrow}  {\mathrm{Aut}}\, (X)
\] 
where the image of the last morphism has finite index. Then:

\begin{theorem} \label{classieux}
For any $\alpha$ in $K_X^{\perp}$ and any $D$ in $\mathrm{NS}\, (X)$, 
\[
\lambda_{\alpha}^*(D)=D-m\,(D.K_X)\, \alpha+\left\{m\,(D. \alpha)-\frac{m^2}{2} (D.K_X)\, \alpha^2 \right\} K_X.
\] 
\end{theorem}
\begin{proof}
We consider the restriction map $\mathfrak{t} \colon \mathrm{Pic}\, (X) \rightarrow \mathrm{Pic}(\mathfrak{X}/\C(t))$ sending $K_X^{\perp}$ to $\mathrm{Pic}_0(\mathfrak{X}/\C(t))$. Then $\mathfrak{t}({\alpha})$ acts on the curve $\mathfrak{X}$ by translation, and also on the Picard group of $\mathfrak{X}$ by the standard formula
\[
\mathfrak{t}({\alpha})^* (\mathfrak{Z})=\mathfrak{Z}+ \mathrm{deg}\,(\mathfrak{Z})\,  \mathfrak{t}({\alpha}).
\] 
\par \smallskip
Applying this to $\mathfrak{Z}=\mathfrak{t}(D)$ and using the formula $\mathrm{deg}\, \mathfrak{t}(D)=-m\,(D.K_X)$, we get
\[
\mathfrak{t}\left(\mathfrak{\lambda}_{\alpha}^* (D)\right)=\mathfrak{t}(D)-m\, (D.K_X) \, \mathfrak{t}(\alpha).
\]
Hence there exists an integer $n$ such that
\[
\mathfrak{\lambda}_{\alpha}^* (D)=D-m\, (D.K_X)\, \alpha + n\, K_X.
\]
Then
\[
\mathfrak{\lambda}_{\alpha}^* (D)^2=D^2-2m \,(D.K_X)\,(D. \alpha)+m^2\, (D.K_X)^2\,\alpha^2+2n\, (D.K_X).
\]
\par \medskip
We can assume without loss of generality that we have $(D.K_X)\neq 0$ since $\mathrm{Pic}\,(X)$ is spanned by such divisors $D$. Since $\mathfrak{\lambda}_{\alpha}^* (D)^2=D^2$, we get 
\[
n=m\, (D. \alpha) -\frac{m^2}{2} \,(D.K_X) \,\alpha^2.
\]
\end{proof}

\bibliographystyle{plain}
\bibliography{bib}

\begin{thebibliography}{10}

\bibitem{BPVDV}
Wolf~P. Barth, Klaus Hulek, Chris A.~M. Peters, and Antonius Van~de Ven.
\newblock {\em Compact complex surfaces}, volume~4 of {\em Ergebnisse der
  Mathematik und ihrer Grenzgebiete. 3. Folge. A Series of Modern Surveys in
  Mathematics [Results in Mathematics and Related Areas. 3rd Series. A Series
  of Modern Surveys in Mathematics]}.
\newblock Springer-Verlag, Berlin, second edition, 2004.

\bibitem{BKv}
Eric Bedford and Kyounghee Kim.
\newblock Periodicities in linear fractional recurrences: degree growth of
  birational surface maps.
\newblock {\em Michigan Math. J.}, 54(3):647--670, 2006.

\bibitem{BK}
Eric Bedford and Kyounghee Kim.
\newblock Continuous families of rational surface automorphisms with positive
  entropy.
\newblock {\em Math. Ann.}, 348(3):667--688, 2010.

\bibitem{BD2}
J{\'e}r{\'e}my Blanc and Julie D{\'e}serti.
\newblock Degree growth of birational maps of the plane.
\newblock {\em Ann. Sc. Norm. Super. Pisa (to appear)}.

\bibitem{BD}
J{\'e}r{\'e}my Blanc and Julie D{\'e}serti.
\newblock Embeddings of {${\rm SL}(2,\Bbb Z)$} into the {C}remona group.
\newblock {\em Transform. Groups}, 17(1):21--50, 2012.

\bibitem{Milnor}
A.~Bonifant, M.~Lyubich, and S.~Sutherland.
\newblock {\em Frontiers in Complex Dynamics: In Celebration of John Milnor's
  80th Birthday}.
\newblock Princeton Mathematical Series. Princeton University Press, 2014.

\bibitem{CD}
Serge Cantat and Igor Dolgachev.
\newblock Rational surfaces with a large group of automorphisms.
\newblock {\em J. Amer. Math. Soc.}, 25(3):863--905, 2012.

\bibitem{Coble}
Arthur~B. Coble.
\newblock Point sets and allied {C}remona groups. {II}.
\newblock {\em Trans. Amer. Math. Soc.}, 17(3):345--385, 1916.

\bibitem{DG}
Julie D{\'e}serti and Julien Grivaux.
\newblock Automorphisms of rational surfaces with positive entropy.
\newblock {\em Indiana Univ. Math. J.}, 60(5):1589--1622, 2011.

\bibitem{DF}
J.~Diller and C.~Favre.
\newblock Dynamics of bimeromorphic maps of surfaces.
\newblock {\em Amer. J. Math.}, 123(6):1135--1169, 2001.

\bibitem{GIZ}
M.~H. Gizatullin.
\newblock Rational {$G$}-surfaces.
\newblock {\em Izv. Akad. Nauk SSSR Ser. Mat.}, 44(1):110--144, 239, 1980.

\bibitem{GH}
Phillip Griffiths and Joseph Harris.
\newblock {\em Principles of algebraic geometry}.
\newblock Wiley Classics Library. John Wiley \& Sons Inc., New York, 1994.
\newblock Reprint of the 1978 original.

\bibitem{HH}
Brian Harbourne.
\newblock Rational surfaces with infinite automorphism group and no
  antipluricanonical curve.
\newblock {\em Proc. Amer. Math. Soc.}, 99(3):409--414, 1987.

\bibitem{DS}
V.~A. Iskovskikh and I.~R. Shafarevich.
\newblock Algebraic surfaces [ {MR}1060325 (91f:14029)].
\newblock In {\em Algebraic geometry, {II}}, volume~35 of {\em Encyclopaedia
  Math. Sci.}, pages 127--262. Springer, Berlin, 1996.

\bibitem{MM}
Curtis~T. McMullen.
\newblock Dynamics on blowups of the projective plane.
\newblock {\em Publ. Math. Inst. Hautes \'Etudes Sci.}, (105):49--89, 2007.

\end{thebibliography}
\end{document}